%% file: main.tex
\documentclass[11pt, reqno, twoside]{amsart}
\input{preamble}

\input{thm_setup}
\newcommand{\Z}{\mathbb{Z}}

\newcommand{\mb}{\mathbb}
\newcommand{\mc}{\mathcal}

\newcommand{\veps}{\varepsilon}

\renewcommand{\bar}{\overline}

\newcommand{\case}[1]{\begin{cases}#1\end{cases}}

\newcommand\angles[1]{\left \langle #1 \right \rangle}
\newcommand\brax[1]{\left \{ #1 \right \} }

\newcommand\paren[1]{\left ( #1 \right ) }
\newcommand\abs[1]{\left | #1 \right | }
\newcommand\braks[1]{\left [ #1 \right ] }
\newcommand{\ord}{\operatorname{ord}}

\newcommand{\Ga}{\Gamma}

\renewcommand{\th}{^{\text{th}}}

\newcommand{\papertitle}{Product and Quotient Set Dominance in Finite Groups}
\fancypagestyle{runningtitle}{%
  \fancyhf{}%
  \fancyhead[LE,RO]{\thepage}%
  \fancyhead[C]{\textsc{\small \papertitle}}%
}
\begin{document}

\title{\papertitle}

\author{Yuval Amit}
\address{Department of Mathematics, University of Maryland, College Park}
\email{\href{mailto:yamit@umd.edu}{yamit@umd.edu}}

\author{Lucas Chen}
\address{Department of Mathematics, University of Chicago}
\email{\href{mailto:lucasch@uchicago.edu}{lucasch@uchicago.edu}}

\author{Christopher Housholder}
\address{Department of Mathematics, Missouri State University}
\email{\href{mailto:christopherlancehousholder@gmail.com}{christopherlancehousholder@gmail.com}}

\author{Joshua Im}
\address{Department of Mathematics, Texas A\&M University}
\email{\href{mailto:jri@tamu.edu}{jri@tamu.edu}}

\author{Joshua Khan}
\address{Department of Mathematics, University of Rochester}
\email{\href{mailto:joshuakhan1307@gmail.com}{joshuakhan1307@gmail.com}}

\author{Steven J. Miller}
\address{Department of Mathematics, Williams College}
\email{\href{mailto:sjm1@williams.edu}{sjm1@williams.edu}}

\author{Mateo Palomares}
\address{Department of Mathematics, University of Michigan}
\email{\href{mailto:mateop@umich.edu}{mateop@umich.edu}}

\author{Devayani Pradhan}
\address{Department of Mathematics, University of Michigan}
\email{\href{mailto:pradhand@umich.edu}{pradhand@umich.edu}}
\date{\today}

\begin{abstract}
A classic problem in additive number theory asks whether a given subset of integers has more sums or differences. This inquiry naturally extends to finite groups, where one compares the sizes of the product set $AA$ and the quotient set $AA^{-1}$ for a subset $A \subseteq G$. A finite group is considered product-dominant if it contains a higher proportion of subsets where $|AA| > |AA^{-1}|$, and quotient-dominant if the reverse is true. We establish statistical criteria to determine the dominance behavior of large finite groups by analyzing the moments of the random variable $d(A) := |AA| - |AA^{-1}|$. By translating product and quotient relationships into graph-theoretic constructions, we compute these moments through counting independent sets.  Our main result demonstrates that for a growing sequence of finite groups $G_k$ with center size $Z_k$ and involution count $T_k$, the condition $Z_k = o(T_k)$ implies the groups are eventually product-dominant. Conversely, if $T_k \cdot \sqrt{3}^{R_k/Z_k} = o(Z_k)$ (where $R_k$ denotes the number of elements whose square is central), the groups are eventually quotient-dominant. Using these asymptotic conditions, we resolve the Miller--Vissuet conjecture by proving that all dihedral groups $D_n$ for $n \geq 3$ are product-dominant. Furthermore, we classify the dominance behavior of several major families of groups, proving that symmetric groups and non-abelian finite simple groups are product-dominant, while groups of odd order and certain abelian groups are quotient-dominant. Finally, for groups that fall outside our asymptotic conditions, we derive an explicit formula for the first moment utilizing the irreducible characters and conjugacy classes of the group.
\end{abstract}

\maketitle

\vspace{-25pt}
\begin{spacing}{0.9}
\tableofcontents
\end{spacing}

%%%%%%%%%%%%%%%%%%%%%%%%%%%%%%%%%%%%%%%%%%%%%%%%%%%%%%%%%%%%%%%%%%%%%%
\section{Introduction}

\subsection{Background}
Given a subset $A \subseteq \{1,\dots,n\}$, we may compare the sizes of the sumset
\begin{align}
    A+A\ :=\ \{ a+a' : a,a' \in A\}
\end{align}
and the difference set
\begin{align}
    A-A \ :=\ \{a -a' : a,a' \in A\}.
\end{align}
The study of subsets with $|A+A| > |A-A|$ has long been of interest to additive number theorists (see \cite{nathanson_history} for a history of the problem). In general, one expects to have $|A+A|<|A-A|$, as addition is commutative while subtraction is not. Sets with $|A+A| > |A-A|$ are called MSTD (more sums than differences) sets. While MSTD sets do exist (the smallest is attributed to J. H. Conway: $\{0, 2, 3, 4, 7, 11, 12, 14\}$), they typically make up only a small percentage of subsets of $\{1,\dots,n\}$. However, Martin and O'Bryant \cite{martinobryant} showed the surprising fact that if $\rho_n$ is the proportion of subsets of $\{1,\dots,n\}$ which are MSTD, then $\rho_n \geq 2\cdot 10^{-7}$ for all $n \geq 15$. Explicit infinite families of MSTD sets have been constructed by Nathanson \cite{nathanson_mstd}, Hegarty \cite{hegarty_explicit}, Miller, Orosz, and Scheinerman \cite{millerorosz}, and Zhao \cite{zhao_ballot}. More recently, Herrmann \cite{herrmann_small} classified the small MSTD sets in arbitrary fields, showing in particular that no abelian group contains an MSTD set of size five. Hegarty and Miller \cite{hegartymiller} later showed that, under a natural model in which each element is chosen with a probability tending to zero, almost all sets are difference dominated.

A natural extension of the MSTD problem is found in studying the subsets of a finite group $G$. Given a subset $A \subseteq G$, we define the \textbf{product set} and \textbf{quotient set} as
\begin{align}
    AA\ &:=\ \{xy : x,y \in A\},
    \\ AA^{-1}\ &:=\ \{xy^{-1} : x,y \in A\}
\end{align}
respectively. In a non-abelian group one could instead compare $AA$ with the left quotient set $A^{-1}A$; the two quotient sets need not have the same size, and the possible differences between them are studied by Duvivier et al.\ \cite{leftrightquotient}. We use the right quotient set $AA^{-1}$, and our work is largely independent of this choice. A subset $A \subseteq G$ is said to be \textbf{MPTQ} (more products than quotients) if $|AA|>|AA^{-1}|$, \textbf{MQTP} if $|AA|<|AA^{-1}|$, and \textbf{balanced} if $|AA| = |AA^{-1}|$. The situation for groups is different from that of $\{1,\dots,n\}$ for two reasons.
\begin{enumerate}[(a)]
    \item Unlike the interval $\{1,\dots,n\}$, a group is closed under multiplication and inverses.
    \item The intuition about the commutativity of addition versus the non-commutativity of subtraction vanishes for non-abelian groups.
\end{enumerate}
Product sets in non-commutative groups have also been studied from the perspective of growth, where the relevant question is how small $|AA|$ can be relative to $|A|$; see Tao \cite{tao_productset}.
Nathanson \cite{nathanson_history} studied product/quotient dominance in abelian groups to find methods that could be applied to the classical MSTD problem. He showed that MSTD sets in finite abelian groups could be lifted to construct infinite families of MSTD sets in the integers. Zhao \cite{ZhaoGroups} subsequently found an asymptotic formula for the proportion of MSTD sets in finite abelian groups, and showed that if $\{G_n\}$ is a sequence of finite abelian groups whose order goes to infinity, the proportion of subsets that are MSTD converges to zero; see also \cite{ZhaoMissing} for the companion analysis of which sums and differences are missing. For the cyclic groups $\Z/n\Z$, Hemmady, Lott, and Miller \cite{hemmadylottmiller} determined when almost all subsets are difference dominated. Miller and Vissuet \cite{miller_vissuet} generalized this to all finite groups by proving that, as the size of a finite group gets arbitrarily large, the proportion of subsets which are balanced converges to $1$. Moreover, the probability that the product set and quotient set both equal the entire group $G_n$ converges to $1$.

Recent research has pivoted to studying the behavior of the small proportion of sets that remain unbalanced. This includes exploring how the structure of non-abelian groups affects the behavior of unbalanced sets. Notably, Miller and Vissuet \cite{miller_vissuet} conjectured the following behavior of dihedral groups. We say that a group is product-dominant if it has more MPTQ sets than MQTP sets, and quotient-dominant if it has more MQTP than MPTQ sets.

\begin{restatable}[Miller--Vissuet]{conj}{millervissuet} \label{conj:miller_vissuet}
All dihedral groups $D_n$ for $n \geq 3$ are product-dominant.
\end{restatable}

Note that $D_n$ is balanced for $n=1,2$. Haviland et al. \cite{haviland} and Ascoli et al. \cite{generalizeddihedral} studied generalized dihedral groups, providing strong evidence in favor of \Cref{conj:miller_vissuet}. We continue along this line by characterizing when large finite groups are product- or quotient-dominant.

%%%%%%%%%%%%%%%%%%%%%%%%%%%%%%%%%%%%%%%%%%%%%%%%%%%%%%%%%%%%%%%%%%%%%%
\subsection{Main results}

For a subset $A \subseteq G$, let
\begin{align}
    d(A)\ :=\ |AA| - |AA^{-1}|.
\end{align}
If $A\subseteq G$ is chosen uniformly at random, we aim to understand the distribution of the random variable $d(A)$. It is therefore natural to study the moments of $d(A)$.
\begin{defn}
    The $r$-th \textbf{raw moment} of $d(A)$ is defined by
\begin{align}
   \mu_r(G)\ :=\ \frac{1}{2^{|G|}}\sum_{A \subseteq G}(d(A))^r. 
\end{align}
For ease of notation, we write
\begin{align}
    M_r(G)\ :=\ \sum_{A \subseteq G}(d(A))^r\ =\ 2^{|G|}\mu_r(G).
\end{align}
Another useful statistic is the number of subsets which are \textbf{balanced}, meaning $d(A)=0$. Let
\begin{align}
    B(G)\ :=\ \# \{A \subseteq G: d(A)=0\}.
\end{align}
We occasionally write $M_r:=M_r(G)$ and $B:=B(G)$.
\end{defn}

The statistical properties of moments (the first two moments in particular) and the number of balanced subsets can help us understand when a group is product- or quotient-dominant. This idea is captured by the following key theorem.

\begin{restatable}{thm}{domineq} \label{thm:dom_ineq}
Let $G$ be a finite group, and suppose
\begin{align}
    \frac{M_1^2}{(2^{|G|}-B)M_2}\ >\ \frac 12. \label{eq:dom_ineq}
\end{align}
Then
\begin{enumerate}[(i)]
    \item $M_1>0 \ \implies\ G$ is product-dominant,
    \item $M_1<0\ \implies\ G$ is quotient-dominant.
\end{enumerate}
\end{restatable}
Our main result uses \Cref{thm:dom_ineq} to determine whether sufficiently large groups, under certain assumptions, are product- or quotient-dominant.

\begin{defn}
    If $\{a_k\},\{b_k\}$ are two sequences, we say $a_k = o(b_k)$ if $\lim_{k \to \infty}{a_k}/{b_k}=0$.
\end{defn}

\begin{restatable}{thm}{mptqthm}
\label{thm:mptqresult}
    Let $\{G_k\}$ be a sequence of finite groups with $|G_k| \to \infty$. Let $Z(G_k)$ denote the center of $G_k$, and define
    \begin{align}
        T_k\ &:=\ \# \{g \in G_k : \ord(g)=2 \}, \\
        Z_k\ &:=\ |Z(G_k)|, \\
        R_k\ &:=\ \#\{ g \in G_k : g^2 \in Z(G_k)\}.
    \end{align}
    Then, we have the following properties.
    \begin{enumerate}[(i)]
        \item If $Z_k = o(T_k)$, then $G_k$ is product-dominant for sufficiently large $k$.
        \item If $T_k\cdot \sqrt 3^{R_k/Z_k} = o(Z_k)$, then $G_k$ is quotient-dominant for sufficiently large $k$.
    \end{enumerate}
\end{restatable}

This establishes that in many cases, the size of the center and the number of order-two elements of a group are what most heavily influence its dominance behavior. There are many notable families of groups which satisfy the conditions in \Cref{thm:mptqresult}. Some of these groups are studied in \Cref{sec:proofandcors}, where we show that sufficiently large
\begin{itemize}
    \item dihedral groups are product-dominant (\Cref{thm:dihedralgroups}),
    \item symmetric groups are product-dominant (\Cref{thm:symmetricgroups}),
    \item non-abelian finite simple groups are product-dominant (\Cref{thm:simplegroups}),
    \item groups of odd order are quotient-dominant (\Cref{thm:oddgroups}), and
    \item abelian groups with at least $6$ squares and at least $171$ elements\footnote{This threshold can be lowered. See \Cref{rmk:abelianthreshold}.} are quotient-dominant (\Cref{thm:abeliangroups}).
\end{itemize}

For dihedral groups in particular, we evaluate precisely when ``sufficiently large'' occurs. Then, with manual computation we confirm all smaller dihedral groups to prove \Cref{conj:miller_vissuet}. We also perform this threshold analysis on abelian groups.

For groups that do not satisfy the conditions of \Cref{thm:mptqresult} (dicyclic groups, for example), deeper study of the moments of $d(A)$ is needed. Computation of $M_1(G)$ in particular is important, as one must find whether \eqref{eq:dom_ineq} is satisfied, and if $M_1>0$ or $M_1<0$. In \Cref{thm:num_of_cycles_prod,thm:expthm}, we establish an explicit formula for $M_1(G)$ based on the irreducible characters and the conjugacy classes of $G$.

Finally, we discuss future work, and in particular we mention special cases where $M_1(G)>0$ while $G$ is quotient-dominant, or $M_1(G)<0$ while $G$ is product-dominant. Groups with these properties cannot be studied via the methods of \Cref{thm:dom_ineq}, making them potentially necessary and interesting to classify.

The paper is organized as follows. \Cref{sec:dom_ineq} proves \Cref{thm:dom_ineq}. \Cref{sec:M1,sec:M2} express $M_1$ and $M_2$ in terms of independent sets of forbiddance graphs, and bounds those counts. \Cref{sec:dom_thm} combines the bounds to prove \Cref{thm:mptqresult}. \Cref{sec:proofandcors} applies \Cref{thm:mptqresult} to dihedral, symmetric, simple, odd-order, and abelian groups. \Cref{sec:proving_expthm} derives a character-theoretic formula for $M_1(G)$, and \Cref{sec:futurework} discusses open problems.

%%%%%%%%%%%%%%%%%%%%%%%%%%%%%%%%%%%%%%%%%%%%%%%%%%%%%%%%%%%%%%%%%%%%%%
\section{\texorpdfstring{Proof of \Cref{thm:dom_ineq}}{Proof of Dominance Inequality}} \label{sec:dom_ineq}

\Cref{thm:dom_ineq}, which is restated below, is a purely statistical fact, but it turns out to be very effective for studying product/quotient-dominant behavior in finite groups.

\domineq*

\begin{proof}
    We prove statement (i), and the proof (ii) is nearly identical, hence omitted. To prove $G$ is product-dominant, we must show
\begin{align} \label{eq:mptq_prob_cond}
    P(d(A)>0)\ >\ P(d(A)<0),
\end{align}
where $P$ denotes the probability over a uniformly chosen random subset $A \subseteq G$.  We use the shorthand $P(d<0)$ and $P(d>0)$.
Let $\mu_1$ be the expected value of $d(A)$ and $\sigma^2$ be the variance. Since $M_1>0$ implies $\mu_1>0$, Cantelli's inequality \cite{cantelli} gives 
\begin{align}
    P(d\leq 0)\ =\ P(d - \mu_1\leq -\mu_1)\ \leq\ \frac{\sigma^2}{\sigma^2 +\mu_1^2}.
\end{align}
Hence
\begin{align}
    P(d>0)\ =\ 1-P(d\leq 0)\ \geq\ 1-\frac{\sigma^2}{\sigma^2 + \mu_1^2}.
\end{align}
On the other hand, let $\beta := P(d=0)=B/2^{|G|}$ be the proportion of balanced sets. Then,
\begin{align}
    P(d<0)\ &=\ 1-P(d\geq 0)\ =\ 1-P(d=0)-P(d>0) \notag
    \\ &=\ 1-\beta-P(d>0)\ \leq\ 1- \beta -\paren{1-\frac{\sigma^2}{\sigma^2 + \mu_1^2}} \notag
    \\ &=\ \frac{\sigma^2}{\sigma^2 + \mu_1^2}-\beta.
\end{align}
Therefore, from \eqref{eq:mptq_prob_cond} we see 
\begin{align} \label{eq:probv1}
    1-\frac{\sigma^2}{\sigma^2 + \mu_1^2} \ >\  \frac{\sigma^2}{\sigma^2  +\mu_1^2}-\beta \ \implies\ G \text{ is product-dominant}. 
\end{align}

Now, we may simplify the expression in \eqref{eq:probv1} by using the fact that $\sigma^2 + \mu_1^2 = \mu_2$ (see \cite{statproofbookSecondMoment}, for example). This gives
\begin{align}
1-\frac{\sigma^2}{\sigma^2 + \mu_1^2}\ >\  \frac{\sigma^2}{\sigma^2  +\mu_1^2}-\beta \ &\iff\ 1+\beta\ >\  \frac{2 \sigma^2}{\sigma^2 + \mu_1^2} \notag
    \\ &\iff 1 +\beta \ >\   \frac{2\mu_2 - 2\mu_1^2}{\mu_2} \notag
    \\ &\iff \frac{2\mu_1^2}{\mu_2}\ >\ 1 -\beta \notag
    \\ &\iff \frac{2M_1^2}{M_2} \ >\ 2^{|G|} -B \notag
    \\ &\iff \frac{M_1^2}{(2^{|G|}-B)M_2}\ >\ \frac 12,
\end{align}
as desired.
\end{proof}

With \Cref{thm:dom_ineq} in mind, we aim to find sufficient lower bounds for $M_1$, and upper bounds for $M_2$ and $2^{|G|}-B$. 

%%%%%%%%%%%%%%%%%%%%%%%%%%%%%%%%%%%%%%%%%%%%%%%%%%%%%%%%%%%%%%%%%%%%%%
\section{\texorpdfstring{The first moment $M_1$}{The First Moment M1}} \label{sec:M1}

To study $M_1(G)$, we express it with a graph-theoretic summation, following the work of \cite{ZhaoGroups}. For a set $S \subseteq G$, let $\mb 1_S$ be the indicator function
\begin{align}
    \mb 1_S(g)\ :=\ \case{1 & g \in S \\ 0 & g \not \in S.}
\end{align}
Then, we use the fact that
\begin{align}
    |S| \ =\ \sum_{g \in G}\mb 1_{S}(g)
\end{align}
to obtain the equality
\begin{align}
    M_1(G)\ =&\ \ \sum_{A \subseteq G}d(A) \ =\ \sum_{A \subseteq G}\paren{|AA| - |AA^{-1}|} \notag
    \\=&\ \ \sum_{A \subseteq G}\sum_{g \in G}(\mb 1_{AA}(g) - \mb 1_{AA^{-1}}(g)) \notag
    \\ =&\ \ \sum_{g \in G}\sum_{A \subseteq G}\paren{\mb 1_{AA}(g) - \mb 1_{AA^{-1}}(g)} \notag
    \\ =&\ \ \sum_{g \in G}\braks{\# \{ A \subseteq G:g \in AA\} -\# \{ A \subseteq G : g \in AA^{-1}\} } \notag
    \\ =&\ \ \sum_{g \in G}\braks{\paren{2^{|G|} - \# \{ A: g \not \in AA\}} - \paren{2^{|G|} - \# \{A:g \not \in AA^{-1}\}}} \notag
    \\ =&\ \ \sum_{g \in G}\braks{\#\{A:g \not \in AA^{-1}\} - \# \{ A : g \not \in AA\}}. \label{eq:exp_breakdown}
\end{align}
To compute the two cardinalities in \eqref{eq:exp_breakdown}, we turn to a graph-theoretic construction first explored in \cite{ZhaoGroups}.

\begin{defn} \label{def:forb_1}
Let $G$ be a group. A \textbf{forbiddance graph} is a graph whose vertices are the elements of $G$, with an edge $x \to y$ whenever $x$ and $y$ satisfy a predetermined product or quotient relationship with respect to some fixed element(s) of $G$. In particular, for fixed $g \in G$, we define the forbiddance graphs $\mc F_P(g)$ and $\mc F_Q(g)$ by the following rules.
\begin{align}
    &\mc F_P(g) \text{ has } x \to y \text{ whenever } xy=g ,
    \\ &\mc F_Q(g) \text{ has } x \to y \text{ whenever } xy^{-1}=g .
\end{align}
\end{defn}

\begin{defn}
    For a directed or undirected graph $\Gamma=(V,E)$, a set $A\subseteq V$ is said to be an \textbf{independent set} if there are no $x,y \in A$ (not necessarily distinct) joined by an edge. The number of independent sets of $\Gamma$ is denoted $i(\Gamma)$.
\end{defn}

\begin{ex}
    Let $G = S_4$ and $g = (1\:2\:3)$. \Cref{fig:FP_S4} depicts $\mc F_P(g)$, with an example independent set boxed and colored in red.
\end{ex}

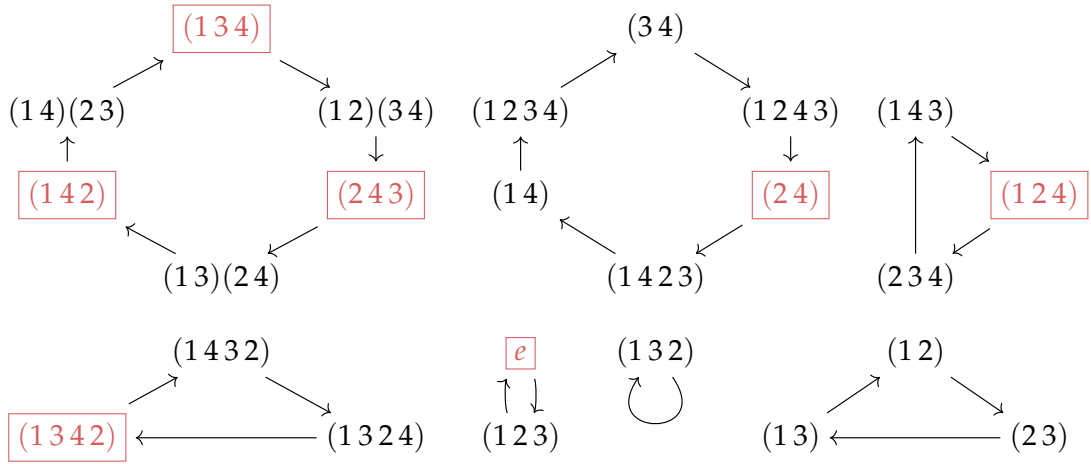
\begin{figure}[ht]
% https://q.uiver.app/#q=WzAsMjQsWzEsMCwiXFxib3hlZHsoMVxcOjNcXDo0KX0iLFswLDYwLDYwLDFdXSxbMCwxLCIoMVxcOjQpKDJcXDozKSJdLFswLDIsIlxcYm94ZWR7KDFcXDo0XFw6Mil9IixbMCw2MCw2MCwxXV0sWzEsMywiKDFcXDozKSgyXFw6NCkiXSxbMiwyLCJcXGJveGVkeygyXFw6NFxcOjMpfSIsWzAsNjAsNjAsMV1dLFsyLDEsIigxXFw6MikoM1xcOjQpIl0sWzQsMCwiKDNcXDo0KSJdLFszLDEsIigxXFw6MlxcOjNcXDo0KSJdLFszLDIsIigxXFw6NCkiXSxbNCwzLCIoMVxcOjRcXDoyXFw6MykiXSxbNSwyLCJcXGJveGVkeygyXFw6NCl9IixbMCw2MCw2MCwxXV0sWzUsMSwiKDFcXDoyXFw6NFxcOjMpIl0sWzcsMiwiXFxib3hlZHsoMVxcOjJcXDo0KX0iLFswLDYwLDYwLDFdXSxbNiwxLCIoMVxcOjRcXDozKSJdLFs2LDMsIigyXFw6M1xcOjQpIl0sWzEsNCwiKDFcXDo0XFw6M1xcOjIpIl0sWzAsNSwiXFxib3hlZHsoMVxcOjNcXDo0XFw6Mil9IixbMCw2MCw2MCwxXV0sWzIsNSwiKDFcXDozXFw6MlxcOjQpIl0sWzYsNCwiKDFcXDoyKSJdLFs1LDUsIigxXFw6MykiXSxbNyw1LCIoMlxcOjMpIl0sWzMsNCwiXFxib3hlZHtlfSIsWzAsNjAsNjAsMV1dLFszLDUsIigxXFw6MlxcOjMpIl0sWzQsNCwiKDFcXDozXFw6MikiXSxbMSwwXSxbMiwxXSxbMywyXSxbNCwzXSxbNSw0XSxbMCw1XSxbNyw2XSxbOCw3XSxbOSw4XSxbMTAsOV0sWzExLDEwXSxbNiwxMV0sWzEzLDEyXSxbMTQsMTNdLFsxMiwxNF0sWzE2LDE1XSxbMTcsMTZdLFsxNSwxN10sWzE5LDE4XSxbMjAsMTldLFsxOCwyMF0sWzIyLDIxLCIiLDIseyJjdXJ2ZSI6LTF9XSxbMjEsMjIsIiIsMix7ImN1cnZlIjotMX1dLFsyMywyMywiIiwwLHsiYW5nbGUiOi0xODB9XV0=
\[\begin{tikzcd}[ampersand replacement=\&,cramped,column sep=tiny, row sep=small]
	\& \textcolor{rgb,255:red,214;green,92;blue,92}{{\boxed{(1\:3\:4)}}} \&\&\& {(3\:4)} \&\&\& \\
	{(1\:4)(2\:3)} \&\& {(1\:2)(3\:4)} \& {(1\:2\:3\:4)} \&\& {(1\:2\:4\:3)} \& {(1\:4\:3)} \\
	\textcolor{rgb,255:red,214;green,92;blue,92}{{\boxed{(1\:4\:2)}}} \&\& \textcolor{rgb,255:red,214;green,92;blue,92}{{\boxed{(2\:4\:3)}}} \& {(1\:4)} \&\& \textcolor{rgb,255:red,214;green,92;blue,92}{{\boxed{(2\:4)}}} \&\& \textcolor{rgb,255:red,214;green,92;blue,92}{{\boxed{(1\:2\:4)}}} \\
	\& {(1\:3)(2\:4)} \&\&\& {(1\:4\:2\:3)} \&\& {(2\:3\:4)} \\
	\& {(1\:4\:3\:2)} \&\& \textcolor{rgb,255:red,214;green,92;blue,92}{{\boxed{e}}} \& {(1\:3\:2)} \&\& {(1\:2)} \\
	\textcolor{rgb,255:red,214;green,92;blue,92}{{\boxed{(1\:3\:4\:2)}}} \&\& {(1\:3\:2\:4)} \& {(1\:2\:3)} \&\& {(1\:3)} \&\& {(2\:3)}
	\arrow[from=1-2, to=2-3]
	\arrow[from=1-5, to=2-6]
	\arrow[from=2-1, to=1-2]
	\arrow[from=2-3, to=3-3]
	\arrow[from=2-4, to=1-5]
	\arrow[from=2-6, to=3-6]
	\arrow[from=2-7, to=3-8]
	\arrow[from=3-1, to=2-1]
	\arrow[from=3-3, to=4-2]
	\arrow[from=3-4, to=2-4]
	\arrow[from=3-6, to=4-5]
	\arrow[from=3-8, to=4-7]
	\arrow[from=4-2, to=3-1]
	\arrow[from=4-5, to=3-4]
	\arrow[from=4-7, to=2-7]
	\arrow[from=5-2, to=6-3]
	\arrow[curve={height=-6pt}, from=5-4, to=6-4]
	\arrow[from=5-5, to=5-5, loop, in=235, out=305, distance=10mm]
	\arrow[from=5-7, to=6-8]
	\arrow[from=6-1, to=5-2]
	\arrow[from=6-3, to=6-1]
	\arrow[curve={height=-6pt}, from=6-4, to=5-4]
	\arrow[from=6-6, to=5-7]
	\arrow[from=6-8, to=6-6]
\end{tikzcd}\]
\caption{The forbiddance graph $\mc F_P((1\:2\:3))$ in $S_4$. The boxed red vertices form an independent set.}
\label{fig:FP_S4}
\end{figure}

The number of independent sets of a graph is, in some cases, well understood (see \cite{ZHAO_indep_sets,ZhaoGroups,zhao_strong_indep}). Therefore, graph-theoretic results on the number of independent sets can be applied to our problem via the following identities.

\begin{rmk}
    For a finite group $G$, and $g \in G$,
    \begin{align}
        i(\mc F_P(g))\ &=\ \# \{A \subseteq G: g \not \in AA\}, \label{eq:prod_indep_id}
    \\ i(\mc F_Q(g))\ &=\ \# \{A \subseteq G: g \not \in AA^{-1}\} .\label{eq:quot_indep_id}
    \end{align}
\end{rmk}

These equalities hold because a subset $A \subseteq G$ satisfies $g \notin AA$ (resp. $g \notin AA^{-1}$) exactly when $A$ contains no $x,y$ (not necessarily distinct) with $xy=g$ (resp. $xy^{-1}=g$); that is, exactly when $A$ is an independent set of $\mc F_P(g)$ (resp. $\mc F_Q(g)$). We therefore obtain the following expression for $M_1(G)$.

\begin{lem} \label{lem:M1_sum}
    For a finite group $G$, we have
    \begin{align}
        M_1(G)\ =\ \sum_{g \in G}\paren{i(\mc F_Q(g)) - i(\mc F_P(g))}. \label{eq:M1_sum}
    \end{align}
\end{lem}

\begin{proof}
    This follows from \eqref{eq:exp_breakdown}, \eqref{eq:prod_indep_id}, and \eqref{eq:quot_indep_id}.
\end{proof}

We now study a few properties of $\mc F_{Q}(g)$ and $\mc F_P(g)$ to aid our later analysis.

\begin{lem} \label{lem:disjoint_cycles}
    For any $g \in G$, both $\mc F_Q(g)$ and $\mc F_P(g)$ are disjoint unions of cycle graphs.
\end{lem}

\begin{proof}
    For $x \in G$, define
    \begin{align}
        q_g(x)\ :=\ q_g^{(1)}(x)\ :=\ g^{-1}x, \quad \quad &q_g^{(n)}(x)\ :=\ q_g(q_g^{(n-1)}(x)),
        \\ p_g(x)\ :=\ p_g^{(1)}(x)\ :=\ x^{-1}g, \quad \quad &p_g^{(n)}(x)\ :=\ p_g(p_g^{(n-1)}(x)).
    \end{align}
    Then $y=q_g(x)$ is the unique element satisfying $xy^{-1}=g$, and similarly $y=p_g(x)$ is the unique element satisfying $xy=g$. Thus, the forbiddance graph $\mc F_Q(g)$ locally around $x$ looks like
    \begin{align}
        x \ \to\ q_g^{(1)}(x)\ \to\ q_g^{(2)}(x)\ \to\ \cdots, \label{eq:FQ_path}
    \end{align}
    and this path must eventually loop back around to $x$, since $q_g$ is a bijection and $G$ is finite. The same holds for $p_g$. It follows that both forbiddance graphs are composed of disjoint cycles.
\end{proof}

The functions $q_g^{(n)}(x)$ and $p_g^{(n)}(x)$ play a central role later, so we record explicit formulas for them.
\begin{lem} \label{lem:pgqg_forms}
    For any $n \geq 1$,
    \begin{align}
        q_g^{(n)}(x)\ =\ g^{-n}x. \label{eq:qg_form}
    \end{align}
    Moreover, 
    \begin{align}
        p_g^{(2n-1)}(x)\ &=\ g^{-n+1}x^{-1}g^{n}, \label{eq:pg_odd_form}
        \\ p_g^{(2n)}(x)\ &= \  g^{-n}xg^n .\label{eq:pg_even_form}
    \end{align}
\end{lem}

\begin{proof}
We proceed by induction. The base case for $q_g^{(1)}$ is by definition, and then
\begin{align}
    q_g^{(n)}(x)\ =\ q_g^{(1)}(q_g^{(n-1)}(x))\ =\ g^{-1}(g^{-(n-1)}x)\ =\ g^{-n}x.
\end{align}
For $p_g$, the base case $p_g^{(1)}$ is by definition. Then
\begin{align}
    p_g^{(2n)}(x)\ &=\ p_g(p_g^{(2n-1)}(x))\ =\ p_g(g^{-n+1}x^{-1}g^n) \notag
    \\ &=\ (g^{-n+1}x^{-1}g^n)^{-1}g\ =\ g^{-n}xg^n,
\end{align}
and
\begin{align}
p^{(2n+1)}_g(x)\ &=\ p_g(p_g^{(2n)}(x))\ =\ p_g(g^{-n}xg^n) \notag
\\ &=\ (g^{-n}xg^n)^{-1}g\ =\ g^{-n}x^{-1}g^{n+1},
\end{align}
as desired.
\end{proof}

\begin{rmk} \label{rmk:lucas_forbiddance}
It is known that the number of independent sets of a cycle graph of length $n$ is the $n\th$ Lucas number\footnote{$L_1=1$, $L_2=3$, and $L_{n+2}=L_{n+1}+L_n$ for $n \geq 1$.} $L_n$, where a cycle of length $1$ is a single vertex with a loop and a cycle of length $2$ is a single edge. Moreover, if $\Gamma := \Gamma_1 \sqcup \Gamma_2$ is the disjoint union of two component graphs, then 
\begin{align} \label{eq:indeps_in_disj_graphs}
    i(\Gamma)\ =\ i(\Gamma_1)\cdot i(\Gamma_2).
\end{align}
See \cite{ZhaoGroups} for proofs. Thus, to find $i(\mc F_P(g))$ and $i(\mc F_Q(g))$, we must compute the lengths of their component cycles.
\end{rmk}

We first analyze $\mc F_Q(g)$ by restating \cite[Lemma 3.3]{ZhaoGroups}. We also reprove it for completeness, as the proof is short and the strategy is important.
\begin{lem} \label{lem:FQ_form}
    For a finite group $G$, and $g \in G$,
    \begin{align}
        i(\mc F_{Q}(g))\ =\ L_{\ord(g)}^{|G|/\ord(g)}.
    \end{align}
\end{lem}

\begin{proof}
    For any element $x \in G$, the length of the cycle in $\mc F_Q(g)$ containing $x$ is precisely the minimal $\ell \geq 1$ such that $q_g^{(\ell)}(x) = x$. By \Cref{lem:pgqg_forms}, this equates to the minimal $\ell$ for which $g^{-\ell}x=x$, which is $\ord (g)$. Since every vertex in $\mc F_Q(g)$ is therefore contained in a cycle of length $\ord(g)$, it follows that $\mc F_Q(g)$ is precisely composed of $|G|/\ord(g)$ cycles of length $\ord (g)$, so the claim follows by \Cref{rmk:lucas_forbiddance} and \eqref{eq:indeps_in_disj_graphs}.
\end{proof}

To understand the behavior of $i(\mc F_{Q}(g))$ (and later $i(\mc F_P(g))$), we aim to know when it is maximized (i.e., what graph configurations yields the highest number of independent sets). For this, we have the following lemma.

\begin{lem} \label{lem:lucas_asym}
    Let $N$ be a positive integer, and let $a_1,\dots,a_j$ be positive integers with $a_i\neq 2$ for all $i$, and $a_1+\cdots + a_j = N$. Then
    \begin{align}
        L_{a_1}\cdot L_{a_2}\cdots L_{a_j}\ \leq \ 7^{N/4}.
    \end{align}
\end{lem}

\begin{proof}
Let $\phi = (1+\sqrt 5)/2$ be the golden ratio. It is well known that
\begin{align}
    L_n\ =\ \phi^n + (1-\phi)^n,
\end{align}
where notably $|1-\phi|<1$, so $L_n \sim \phi^n$. Moreover, since $1-\phi <0$, one sees that $L_n > \phi^n \iff n$ is even. Thus, as $L_4=7$, and $(1-\phi)^n < (1-\phi)^4$ for all $n>4$, we get that
\begin{align}
    L_n\ \leq\ \paren{7^{1/4}}^n
\end{align}
for all $n \geq 4$. This inequality also holds for $L_1$ and $L_3$. As $a_i \neq 2$ for all $i$, we have
\begin{align}
    \prod_{i=1}^jL_{a_i}\ \leq \ \prod_{i=1}^j 7^{a_i/4}\ =\ 7^{N/4}.
\end{align}
% Finally, we compute that $7^{1/4}<3^{1/2}$, hence
% \begin{align}
%     \lim_{N \to \infty}\frac{7^{N/4}}{3^{N/2}}=0.
% \end{align}
\end{proof}

The exclusion of $a_i=2$ is important here, as a graph with $N/2$ two-cycles has $L_2^{N/2}=\sqrt 3^{\,N}$ independent sets, which is larger than $7^{N/4}$ (in the sense of little-o) since $7^{1/4}<\sqrt 3$. We use this fact to show that, loosely speaking, group elements $g$ with $\mc F_Q(g)$ or $\mc F_P(g)$ composed almost entirely of two-cycles have the greatest contribution in terms of number of independent sets.

\begin{cor} \label{cor:FQ_asym}
    Let $G$ be a finite group of size $N$, and let $T$ be the number of elements of $G$ with order $2$. Then
    \begin{align}
        T \cdot \sqrt 3^{N}\ \leq\ \sum_{g \in G}i(\mc F_Q(g))\ \leq\ T \cdot \sqrt 3^{N} + (N-T)\cdot 7^{N/4}. \label{eq:FQ_asym}
    \end{align}
\end{cor}

\begin{proof}
    By \Cref{lem:FQ_form},
    \begin{align}
        \sum_{g \in G}i(\mc F_Q(g))\ &=\ \sum_{g \in G}L_{\ord(g)}^{N/\ord(g)}\ \notag 
        \\ &=\ \, \sum_{\mathclap{\substack{g \in G \\ \ord(g)=2}}} \:L_2^{N/2} + \: \: \sum_{\mathclap{\substack{g \in G \\ \ord (g) \neq 2}}} \: L_{\ord(g)}^{N/\ord(g)} \notag
        \\ &=\ T \sqrt 3^N + \: \: \sum_{\mathclap{\substack{g \in G \\ \ord(g) \neq 2}}}\: L_{\ord(g)}^{N/\ord(g)}.
    \end{align}
    For each of the $N-T$ elements $g$ with $\ord(g)\neq 2$, the graph $\mc F_Q(g)$ consists of $N/\ord(g)$ cycles of length $\ord (g) \neq 2$, and these lengths sum to $N$, so $L_{\ord(g)}^{N/\ord(g)} \leq 7^{N/4}$ by \Cref{lem:lucas_asym}. The remaining summation is therefore at most $(N-T)\cdot 7^{N/4}$, and it is clearly non-negative, which gives both bounds in \eqref{eq:FQ_asym}.
\end{proof}

Understanding $i(\mc F_P(g))$ is more difficult, since the length of the cycle containing $x\in G$ is the minimal $\ell\geq 1$ with $p_g^{(\ell)}(x)=x$, and by \eqref{eq:pg_odd_form} and \eqref{eq:pg_even_form} this depends on both $x$ and $g$ in a less uniform way. However, the number of independent sets is still a product of Lucas numbers whose indices sum to $|G|$, as in \Cref{lem:lucas_asym}. Since \Cref{cor:FQ_asym} demonstrated that the $2$-cycles give the dominant contribution, we want to study the $2$-cycles in $\mc F_P(g)$ as well. 

\begin{lem} \label{lem:FP_asym}
    Let $G$ be a finite group of size $N$ with center $Z(G)$. Let $Z:=|Z(G)|$, and let
    \begin{align}
        R\ :=\ \#\{x \in G: x^2 \in Z(G)\}.
    \end{align}
    Then
    \begin{align}
       Z\cdot \sqrt 3^{N-R/Z} \ \leq \ \sum _{g \in G}i(\mc F_P(g)) \ \leq\ Z\cdot \sqrt 3^N + (N-Z)\cdot 1.68^N. \label{eq:FP_asym}
    \end{align}
    Moreover, for every $g \not \in Z(G)$,
    \begin{align}
        i(\mc F_P(g))\ \leq\ \paren{3^{1/4}\cdot 7^{1/8}}^N\ <\ 1.68^N. \label{eq:noncentral_P}
    \end{align}
\end{lem}

\begin{proof}
Suppose $g \in Z(G)$. By \eqref{eq:pg_even_form}, $p_g^{(2)}(x) = g^{-1}xg$. Thus, as $gx=xg$, we know every $x\in G$ is contained in either a $1$-cycle or a $2$-cycle in $\mc F_P(g)$ by the proof of \Cref{lem:disjoint_cycles}. In particular, $x$ is contained in a $2$-cycle if and only if $p_g^{(1)}(x)\neq x$, which is equivalent to $x^2\neq g$. Let 
\begin{align}
    r(g)\ :=\ \# \{x \in G: x^2=g\}.
\end{align}
Then, the forbiddance graph $\mc F_P(g)$ is composed of $r(g)$ $1$-cycles and $(N-r(g))/2$ two-cycles. Therefore, for any $g \in Z(G)$,
\begin{align}
    i(\mc F_P(g))\ =\ L_1^{r(g)}\cdot L_2^{\frac{N-r(g)}{2}}\ =\ \sqrt 3^{N-r(g)}.
\end{align}
On the other hand, suppose $g \not \in Z(G)$. By \eqref{eq:pg_even_form}, $p_g^{(2)}(x)=x$ if and only if $xg=gx$, so every vertex lying in a $2$-cycle of $\mc F_P(g)$ belongs to the centralizer $C_G(g)$. Let $k$ be the number of such vertices. As $g$ is non-central, $C_G(g)$ is a proper subgroup of $G$, hence $k \leq |C_G(g)|\leq N/2$. The $2$-cycles contribute $L_2^{k/2}=\sqrt 3^{\,k}$ independent sets, while the remaining cycles have lengths different from $2$ summing to $N-k$, so by \Cref{lem:lucas_asym} they contribute at most $7^{(N-k)/4}$. Since $\sqrt 3 > 7^{1/4}$, the resulting bound increases with $k$, and therefore
\begin{align}
    i(\mc F_P(g))\ \leq\ \sqrt 3^{\,k}\cdot 7^{(N-k)/4}\ \leq\ \sqrt 3^{\,N/2}\cdot 7^{N/8}\ =\ \paren{3^{1/4}\cdot 7^{1/8}}^N\ <\ 1.68^N,
\end{align}
which proves \eqref{eq:noncentral_P}. Splitting the sum over the central and non-central elements of $G$ now gives
\begin{align}
    \sum_{g \in G}i(\mc F_P(g))\ &=\ \sum_{g \in Z(G)}\sqrt 3^{N-r(g)} + \sum_{g \not \in Z(G)}i(\mc F_P(g)). \label{eq:FP_true_sum}
\end{align}
Since $\sqrt 3^{N-r(g)}\leq \sqrt 3^N$, and since each of the $N-Z$ non-central elements contributes at most $1.68^N$ by \eqref{eq:noncentral_P}, we have
\begin{align}
    \sum_{g \in G}i(\mc F_P(g))\ \leq\ Z\cdot \sqrt 3^N + (N-Z)\cdot 1.68^N,
\end{align}
which is the second inequality in \eqref{eq:FP_asym}. To obtain the lower bound, we discard the non-negative second sum in \eqref{eq:FP_true_sum} and observe that $f(x):=\sqrt 3^{N-x}$ is a convex function. Therefore, by Jensen's inequality \cite{jensen1906, stein2005real},
\begin{align}
    \frac{1}{Z}\sum_{g \in Z(G)}\sqrt 3^{N-r(g)}\ \geq\ \sqrt 3^{ N-(1/Z)\sum_{g \in Z(G)}r(g)}\ =\ \sqrt 3^{N-R/Z}, \label{eq:jensen}
\end{align}
since $\sum_{g \in Z(G)}r(g) = R$. Thus, by \eqref{eq:FP_true_sum},
\begin{align}
    \sum_{g \in G}i(\mc F_P(g))\ \geq\ Z\cdot \sqrt 3^{N-R/Z}
\end{align}
as desired.
\end{proof}

%%%%%%%%%%%%%%%%%%%%%%%%%%%%%%%%%%%%%%%%%%%%%%%%%%%%%%%%%%%%%%%%%%%%%%
\section{\texorpdfstring{The second moment $M_2$}{The Second Moment M2}} \label{sec:M2}

Similar to the first moment, we reduce the study of the second moment to evaluating the number of independent sets of forbiddance graphs.

\begin{defn}
    We define three more forbiddance graphs as in \Cref{def:forb_1}. Once again, the vertices of these graphs are the elements of $G$, and for fixed $g,h \in G$, we define
    \begin{align}
         &\mc F_{PP}(g,h) \text{ has } x \to y \text{ whenever } xy=g \text{ or } xy=h,
    \\ &\mc F_{QQ}(g,h) \text{ has } x \to y \text{ whenever } xy^{-1}=g \text{ or } xy^{-1}=h,
    \\ &\mc F_{PQ}(g,h) \text{ has } x \to y \text{ whenever } xy=g \text{ or } xy^{-1}=h.    
    \end{align}
\end{defn}

\begin{rmk} \label{rmk:two_elm_indep}
   Similar to before, the number of independent sets of these forbiddance graphs relates to cardinalities of subsets satisfying certain element exclusion properties. For a finite group $G$, and $g,h \in G$,
    \begin{align}
 i(\mc F_{PP}(g,h))\ &=\ \# \{A \subseteq G: g \not \in AA, h \not \in AA\},
    \\ i(\mc F_{QQ}(g,h))\ &=\ \# \{A \subseteq G: g \not \in AA^{-1}, h \not \in AA^{-1}\},
    \\ i(\mc F_{PQ}(g,h))\ &=\ \# \{A \subseteq G: g \not \in AA, h \not \in AA^{-1}\}.
\end{align}
\end{rmk}

\begin{lem} \label{lem:2nd_mom}
    Recall $M_2(G):= \sum_{A \subseteq G}(d(A))^2$. For any finite group $G$,
    \begin{align}
        M_2(G)\ =\ \sum_{g,h \in G}\braks{i(\mc F_{PP}(g,h))+ i(\mc F_{QQ}(g,h))- 2 i(\mc F_{PQ}(g,h))}.
    \end{align}
\end{lem}

\begin{proof}
We proceed similarly to the first moment. Let $p(A) := |AA|$ and $q(A):= |AA^{-1}|$. Then
\begin{align}
    M_2(G)\ :=&\ \sum_{A \subseteq G}(p(A) - q(A))^2\ =\ \sum_{A \subseteq G}\braks{p(A)^2+q(A)^2 - 2p(A)q(A)}. \label{eq:binomvar}
\end{align}
Now, note that
\begin{align}
    p(A)^2\ =&\ \ \paren{\sum_{x}\mb 1_{AA}(x)}^2 = \sum_{g,h}\mb 1_{AA}(g)\mb 1_{AA}(h),
    \\ q(A)^2\ =&\ \ \paren{\sum_{x}\mb 1_{AA^{-1}}(x)}^2\ =\ \sum_{g,h}\mb 1_{AA^{-1}}(g)\mb 1_{AA^{-1}}(h),
    \\ p(A)q(A)\ =&\ \ \paren{\sum_{g}\mb 1_{AA}(g)}\paren{\sum_{h}\mb 1_{AA^{-1}}(h)}\ =\ \sum_{g,h}\mb 1_{AA}(g) \mb 1_{AA^{-1}}(h).
\end{align}
Thus, we rewrite \eqref{eq:binomvar} as 
\begin{align}
    &\ \ \sum_{A \subseteq G} \sum_{g,h}\braks{\mb 1_{AA}(g)\mb 1_{AA}(h) + \mb 1_{AA^{-1}}(g)\mb 1_{AA^{-1}}(h) - 2\cdot \mb 1_{AA}(g) \mb 1_{AA^{-1}}(h)} \notag
    \\ =&\ \ \sum_{g,h}\sum_{A \subseteq G}\braks{\mb 1_{AA}(g)\mb 1_{AA}(h) + \mb 1_{AA^{-1}}(g)\mb 1_{AA^{-1}}(h) - 2\cdot \mb 1_{AA}(g) \mb 1_{AA^{-1}}(h)} \notag
    \\ =&\ \ \sum_{g,h}\braks{ \# \{ A: g,h \in AA\} + \#\{A: g, h \in AA^{-1}\} -2 \# \{ A: g \in AA, h \in AA^{-1}\} }. \label{eq:varandsets}
\end{align}

Now, to obtain the expressions in \Cref{rmk:two_elm_indep}, we use the inclusion--exclusion principle to obtain
\begin{align}
    \# \{ A: g,h \in AA\}\ \ =&\ \ \ 2^{|G|}-\# \{ A: g\not \in AA \text{ or } h \not \in AA\}  \notag
    \\ =&\ \ \ 2^{|G|}-\#\{A: g \not \in AA\} - \#\{A: h \not \in AA\} \notag
    \\ &\ \ +\ \#\{ A: g \not \in AA \text{ and } h \not \in AA\} \notag
    \\ =&\ \ \ 2^{|G|}-i(\mathcal F_{P}(g)) - i(\mathcal F_{P}(h))  + i(\mathcal F_{PP}(g,h)), \label{eq:varprod}
\end{align}
and similarly
\begin{align}
    \# \{ A:g,h\in AA^{-1}\}\ =&\ \ 2^{|G|}-i(\mathcal F_Q(g)) - i(\mathcal F_Q(h))  + i(\mathcal F_{QQ}(g,h)), \label{eq:varquot}
    \\ \# \{A: g \in AA,h \in AA^{-1}\}\ =&\ \ 2^{|G|} - i(\mathcal F_P(g)) - i(\mathcal F_Q(h)) + i(\mathcal F_{PQ}(g,h)). \label{eq:varboth}
\end{align}
Plugging these into \eqref{eq:varandsets} yields (after rearranging and canceling)
\begin{align}
    M_2(G)\ =\ \ &\sum_{g,h}\braks{i(\mc F_P(g))- i(\mc F_P(h))+ i(\mc F_Q(h))-i(\mc F_Q(g))} \label{eq:varsingles} 
    \\ &+\ \sum_{g,h}\braks{i(\mc F_{PP}(g,h))+ i(\mc F_{QQ}(g,h))- 2 i(\mathcal F_{PQ}(g,h))}. \label{eq:vardoubles}
\end{align}
Notice that the entire sum in \eqref{eq:varsingles} cancels, and so we are left with
\begin{align}
    M_2(G)\ =\ \sum_{g,h}\braks{i(\mc F_{PP}(g,h))+ i(\mc F_{QQ}(g,h))- 2 i(\mathcal F_{PQ}(g,h))}, \label{eq:M2_final}
\end{align}
as desired.
\end{proof}

We aim to find upper bounds for $M_2(G)$ to satisfy \Cref{thm:dom_ineq}. From \Cref{lem:2nd_mom}, we may rewrite $M_2(G)$ as
\begin{align}
    M_2(G)\ =\ \ &\sum_{g \in G}\braks{i (\mc F_{P}(g)) + i(\mc F_Q(g)) - 2i\paren{\mc F_{PQ}(g,g)}}  \label{eq:M2_singles}
    \\ &+ \sum_{g\neq h \in G}\braks{i(\mc F_{PP}(g,h))+ i(\mc F_{QQ}(g,h))- 2 i(\mc F_{PQ}(g,h))}. \label{eq:M2_doubles}
\end{align}

It suffices for our purposes to remove the negative terms from \eqref{eq:M2_singles} and \eqref{eq:M2_doubles}, so
\begin{align}
    M_2(G)\ \leq \ \ &\sum_{g \in G}i (\mc F_{P}(g)) + \sum_{g \in G}i(\mc F_Q(g)) \label{eq:M2_less_singles}
    \\ &+ \sum_{g\neq h \in G}i(\mc F_{PP}(g,h))+ \sum_{g \neq h \in G}i(\mc F_{QQ}(g,h)). \label{eq:M2_less_doubles}
\end{align}

The behavior of the two sums in \eqref{eq:M2_less_singles} was studied in \Cref{lem:FP_asym} and \Cref{cor:FQ_asym}. To bound the two sums in \eqref{eq:M2_less_doubles}, we show that every term is exponentially smaller than $\sqrt 3^N$. For this, we need a result by Sah, Sawhney, Stoner, and Zhao \cite{zhao_strong_indep}.

\begin{thm}[\cite{zhao_strong_indep}] \label{thm:zhao_strong_indep}
Let $\Gamma=(V,E)$ be a simple, undirected graph with no isolated vertices. Let $d_v$ be the degree of a vertex $v \in V$. Then
    \begin{align}
        i(\Gamma) \ \leq \ \prod_{uv \in E} \paren{2^{d_u}+2^{d_v}-1}^{\frac{1}{d_ud_v}} .  \label{eq:zhao_strong_indep}
    \end{align}
\end{thm}

\begin{cor} \label{cor:Zhao_strong_indep}
    Let $\Gamma = (V,E)$ be a simple, undirected graph with no isolated vertices such that $2 \leq d_v \leq 4$ for all $v \in V$. Then
    \begin{align}
        i(\Gamma)\ \leq\ 19^{|V|/6}. \label{eq:cor_zhao}
    \end{align}
\end{cor}
\begin{proof}
    To find an upper bound for the expression in \eqref{eq:zhao_strong_indep}, note that
    \begin{align}
        \log \paren{\prod_{uv \in E} \paren{2^{d_u}+2^{d_v}-1}^{\frac{1}{d_ud_v}} }\ =\ \sum_{uv \in E}\frac{\log(2^{d_u}+2^{d_v}-1)}{d_ud_v}.
    \end{align}
    We wish to use the graph-theoretic identity
    \begin{align}
        \sum_{uv \in E}\paren{\frac{1}{d_u}+\frac{1}{d_v}}\ =\ |V|,
    \end{align}
    so we want the minimal constant $c$ with
    \begin{align}
        \frac{\log(2^{d_u}+2^{d_v}-1)}{d_ud_v} \ \leq \ c\paren{\frac{1}{d_u}+\frac{1}{d_v}},
    \end{align}
    which is equivalent to
    \begin{align}
        \frac{\log(2^{d_u}+2^{d_v}-1)}{d_u+d_v}\ \leq\ c. \label{eq:logfraclessc}
    \end{align}
    Since $d_u,d_v \in \{2,3,4\}$, one can compute that the maximum of the left-hand side of \eqref{eq:logfraclessc} is attained at $\{d_u,d_v\}=\{2,4\}$, so we set $c=\log(19)/6$. Then
    \begin{align}
        \sum_{uv \in E}\frac{\log(2^{d_u}+2^{d_v}-1)}{d_ud_v}\ \leq\ \frac{\log(19)}{6}\sum_{uv \in E}\paren{\frac{1}{d_u}+\frac{1}{d_v}}\ =\ \frac{\log (19)}{6}|V|,
    \end{align}
    hence
    \begin{align}
        i(\Gamma)\ \leq\ \exp\paren{ \sum_{uv \in E}\frac{\log (2^{d_u}+2^{d_v}-1)}{d_ud_v}}\ \leq\  \exp\paren{\frac{\log (19)}{6}|V|}\ =\ 19^{|V|/6},
    \end{align}
    as desired.
\end{proof}

For the last lemma in this section, we need the number of independent sets of a path. The path $\Pi_n$ with $n \geq 1$ vertices has $F_{n+2}$ independent sets, where $F_1=F_2=1$ and $F_{n+2}=F_{n+1}+F_n$ are the Fibonacci numbers (see \cite{ZhaoGroups}, for example).

\begin{lem} \label{lem:path_bound}
    For every $n \geq 1$,
    \begin{align}
        i(\Pi_n)\ =\ F_{n+2}\ \leq\ 2\cdot 7^{(n-1)/4}. \label{eq:path_bound}
    \end{align}
\end{lem}

\begin{proof}
    For $n=1$ and $n=2$ we have $F_3 = 2 = 2\cdot 7^0$ and $F_4 = 3 < 2 \cdot 7^{1/4}$. For $n \geq 3$, by induction,
    \begin{align}
        F_{n+2}\ &=\ F_{n+1}+F_n\ \leq\ 2\cdot 7^{(n-2)/4} + 2\cdot 7^{(n-3)/4} \notag \\ &=\ 2\cdot 7^{(n-1)/4}\paren{7^{-1/4}+7^{-1/2}} \notag \\ &\leq\ 2\cdot 7^{(n-1)/4},
    \end{align}
    since $7^{-1/4}+7^{-1/2} < 1$.
\end{proof}

\begin{lem} \label{lem:double_bound}
    Let $G$ be a finite group of order $N$, and let $g\neq h \in G$. Then
    \begin{align}
    i(\mc F_{QQ}(g,h))\ &\leq\ 19^{N/6}, \label{eq:FQQ_bound}
    \\ i(\mc F_{PP}(g,h))\ &\leq\ \paren{3^{1/4}\cdot 7^{1/8}}^N\ <\ 1.68^N. \label{eq:FPP_bound}
    \end{align}
    Moreover, if $g,h \in Z(G)$, then
    \begin{align}
        i(\mc F_{PP}(g,h))\ \leq\ 7^{N/4}. \label{eq:FPP_central}
    \end{align}
\end{lem}

\begin{proof}
We first remark on three facts about independent sets. First, a self-adjacent vertex lies in no independent set, so deleting all self-adjacent vertices (and the edges attached to them) does not change the number of independent sets. Second, adding edges to a graph on a fixed vertex set can only decrease the number of independent sets. In particular, since $\mc F_P(g)$ and $\mc F_P(h)$ are subgraphs of $\mc F_{PP}(g,h)$ on the same vertex set,
\begin{align}
    i(\mc F_{PP}(g,h))\ \leq\ \min\brax{i(\mc F_P(g)),\, i(\mc F_P(h))}. \label{eq:PP_le_P}
\end{align}
Third, for any directed graph $\mc F$, we define the underlying undirected graph $\Gamma$ to have the same vertex set as $\mc F$, where each directed edge of $\mc F$ becomes an undirected edge, and two or more directed edges between the same pair of vertices are identified as a single undirected edge; see \Cref{fig:directed_undirected}.

\begin{figure}[ht]
% https://q.uiver.app/#q=WzAsMTAsWzAsMCwiXFxidWxsZXQiXSxbMCwyLCJcXGJ1bGxldCJdLFsyLDIsIlxcYnVsbGV0Il0sWzIsMCwiXFxidWxsZXQiXSxbNCwwLCJcXGJ1bGxldCJdLFs2LDAsIlxcYnVsbGV0Il0sWzYsMiwiXFxidWxsZXQiXSxbNCwyLCJcXGJ1bGxldCJdLFsxLDEsIlxcbWF0aGNhbHtGfSJdLFs1LDEsIlxcR2FtbWEiXSxbMCwzLCIiLDAseyJvZmZzZXQiOjF9XSxbMCwxLCIiLDIseyJvZmZzZXQiOi0xfV0sWzEsMCwiIiwyLHsib2Zmc2V0IjotMX1dLFsxLDJdLFsyLDIsIiIsMix7ImFuZ2xlIjoxMzV9XSxbMCwzLCIiLDIseyJvZmZzZXQiOi0xfV0sWzcsNCwiIiwyLHsic3R5bGUiOnsiaGVhZCI6eyJuYW1lIjoibm9uZSJ9fX1dLFs0LDUsIiIsMix7InN0eWxlIjp7ImhlYWQiOnsibmFtZSI6Im5vbmUifX19XSxbNyw2LCIiLDAseyJzdHlsZSI6eyJoZWFkIjp7Im5hbWUiOiJub25lIn19fV0sWzYsNiwiIiwwLHsiYW5nbGUiOjEzNSwic3R5bGUiOnsiaGVhZCI6eyJuYW1lIjoibm9uZSJ9fX1dXQ==
\[\begin{tikzcd}[ampersand replacement=\&]
	\bullet \&\& \bullet \&\& \bullet \&\& \bullet \\
	\& {\mathcal{F}} \&\&\&\& \Gamma \\
	\bullet \&\& \bullet \&\& \bullet \&\& \bullet
	\arrow[shift right, from=1-1, to=1-3]
	\arrow[shift left, from=1-1, to=1-3]
	\arrow[shift left, from=1-1, to=3-1]
	\arrow[no head, from=1-5, to=1-7]
	\arrow[shift left, from=3-1, to=1-1]
	\arrow[from=3-1, to=3-3]
	\arrow[from=3-3, to=3-3, loop, in=280, out=350, distance=10mm]
	\arrow[no head, from=3-5, to=1-5]
	\arrow[no head, from=3-5, to=3-7]
	\arrow[no head, from=3-7, to=3-7, loop, in=280, out=350, distance=10mm]
\end{tikzcd}\]
\vspace{-20pt}
\caption{An example directed graph $\mc F$, and its underlying undirected graph $\Gamma$.}
\label{fig:directed_undirected}
\end{figure}
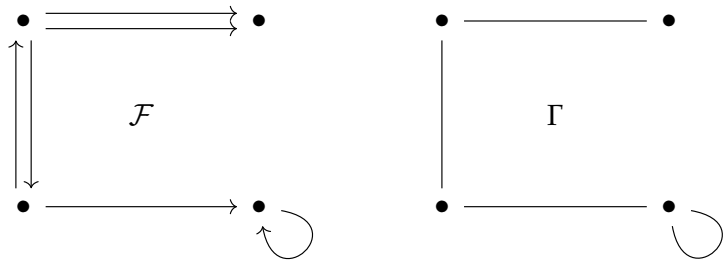

Then it is clear that $i(\mc F)=i(\Gamma)$. Moreover, if $\mc F$ has no self-adjacent vertices, then $\Gamma$ is a simple, undirected graph.

We now prove \eqref{eq:FQQ_bound}. If $g=e$ or $h=e$, then every $x \in G$ is self-adjacent in $\mc F_{QQ}(g,h)$, since $xx^{-1}=e$, so $i(\mc F_{QQ}(g,h))=1$ and the claim follows trivially. Otherwise if $g,h \neq e$, then no vertex is self-adjacent. Let $\Ga_{QQ}$ denote the simple, undirected graph corresponding to $\mc F_{QQ}(g,h)$. The neighbors of $x$ in $\Ga_{QQ}$ are the elements $\{g^{-1}x,\ h^{-1}x,\ gx,\ hx\}$. These are all different from $x$, and $g^{-1}x \neq h^{-1}x$, so every vertex has degree $2$, $3$, or $4$. Thus \eqref{eq:FQQ_bound} follows from \Cref{cor:Zhao_strong_indep}.

To prove \eqref{eq:FPP_bound}, we first assume $g \not \in Z(G)$ (the case $h \not \in Z(G)$ is identical). Then $i(\mc F_P(g)) \leq \paren{3^{1/4}\cdot 7^{1/8}}^N$ by \eqref{eq:noncentral_P}, and so \eqref{eq:FPP_bound} follows from \eqref{eq:PP_le_P}.

Finally, suppose $g,h \in Z(G)$. We then get $xy=g$ if and only if $yx=g$, and the unique $y$ with $xy=g$ is $p_g(x) := x^{-1}g$. The map $p_g$ is an involution, since $p_g(p_g(x)) = g^{-1}xg = x$. Hence, in the underlying undirected graph $\Ga_{PP}$ of $\mc F_{PP}(g,h)$, the neighbors of $x$ are exactly $p_g(x)$ and $p_h(x)$, and these are distinct because $g \neq h$. The vertex $x$ is self-adjacent if and only if 
$x \in S := \{x \in G : x^2 \in \{g,h\}\}$. Let $s := |S|$, and let $\Ga'$ be the simple graph obtained from $\Ga_{PP}$ by deleting $S$, so that $\Ga'$ is simple, undirected, and $i(\mc F_{PP}(g,h)) = i(\Ga')$.

Every vertex of $\Gamma_{PP}$ outside $S$ has exactly two neighbors, both different from itself, while every vertex in $S$ has exactly one neighbor different from itself. In particular, $\Ga'$ has maximum degree at most $2$, so it is a disjoint union of cycles of lengths $c_1,\dots,c_a \geq 3$ and paths with $n_1,\dots,n_b \geq 1$ vertices, where
\begin{equation}
    \sum_i c_i + \sum_j n_j = N-s.
\end{equation}
Consider a path component $\Pi$ with $n$ vertices. If $n>1$, the endpoints of $\Pi$ have degree $1$ as vertices in $\Ga'$, but degree $2$ as vertices in $\Ga_{PP}$, so the neighbor each endpoint loses lies in $S$. If $n=1$, then $\Pi$ is a single vertex $x \not \in S$, which has two distinct neighbors in $\Ga_{PP}$, and both lie in $S$. In either case, each path component of $\Ga'$ accounts for two elements of $S$, and these elements are distinct for distinct path components, since every element of $S$ has exactly one neighbor other than itself. Therefore $2b \leq s$. Therefore, by \Cref{lem:lucas_asym}, \Cref{rmk:lucas_forbiddance}, and \Cref{lem:path_bound},
\begin{align}
    i(\mc F_{PP}(g,h))\ =\ \prod_{i=1}^a L_{c_i}\cdot \prod_{j=1}^b F_{n_j+2} \ &\leq\ 7^{\sum_i c_i/4}\cdot \prod_{j=1}^b \paren{2\cdot 7^{(n_j-1)/4}} \notag
    \\ &=\ 7^{(N-s)/4}\cdot \paren{2\cdot 7^{-1/4}}^b \notag
    \\ &\leq\ 7^{(N-s)/4}\cdot \paren{2\cdot 7^{-1/4}}^{s/2} \notag
    \\ &=\ 7^{N/4}\cdot\paren{\sqrt 2\cdot 7^{-3/8}}^s\ \leq\ 7^{N/4}, \label{eq:central_PP_chain}
\end{align}
where we used $2 \cdot 7^{-1/4} > 1$ and $\sqrt 2 < 7^{3/8}$. This proves \eqref{eq:FPP_central}. Since $7^{1/4} < 3^{1/4}\cdot 7^{1/8}$, this also completes the proof of \eqref{eq:FPP_bound}.
\end{proof}

%%%%%%%%%%%%%%%%%%%%%%%%%%%%%%%%%%%%%%%%%%%%%%%%%%%%%%%%%%%%%%%%%%%%%%
\section{\texorpdfstring{Proof of \Cref{thm:mptqresult}}{Proof of Dominance Theorem}} \label{sec:dom_thm}

We now combine the bounds on $M_1$ from \Cref{sec:M1} with the bounds on $M_2$ from \Cref{sec:M2} and apply \Cref{thm:dom_ineq} to prove \Cref{thm:mptqresult}.
\mptqthm*

\begin{proof}
Let $N_k := |G_k|$. By \Cref{thm:dom_ineq}, to prove that $G_k$ is product- or quotient-dominant for sufficiently large $k$, it suffices to show that
\begin{align}
    \frac{1}{2}\ <\ \frac{(M_1(G_k))^2}{(2^{N_k}-B(G_k))\cdot (M_2(G_k))} \label{eq:dom_goal}
\end{align}
for all sufficiently large $k$, and to determine the sign of $M_1(G_k)$.

\begin{lem} \label{lem:PQPPQQ}
    Let
    \begin{align}
    P_k\ :=\ \sum_{g \in G_k}i(\mc F_{P}(g)), \quad &\quad \quad Q_k \ :=\ \sum_{g \in G_k}i(\mc F_Q(g)),
    \\ PP_k\ :=\ \sum_{g\neq h \in G_k}i(\mc F_{PP}(g,h)), \quad &\quad \quad QQ_k\ :=\ \sum_{g\neq h \in G_k}i(\mc F_{QQ}(g,h)).
\end{align}
Then
\begin{align}
    M_1(G_k)\ &=\ Q_k-P_k, \label{eq:M1_in_terms}
    \\ M_2(G_k)\ &\leq\ Q_k+P_k+QQ_k+PP_k, \label{eq:M2_in_terms}
    \\ 2^{N_k}-B(G_k)\ &\leq\ Q_k+P_k. \label{eq:B_in_terms}
\end{align}
\end{lem}

\begin{proof}
    We immediately obtain \eqref{eq:M1_in_terms} from \Cref{lem:M1_sum}, and \eqref{eq:M2_in_terms} was shown in \eqref{eq:M2_less_singles} and \eqref{eq:M2_less_doubles}. Finally, \eqref{eq:B_in_terms} holds for any finite group $G$, as
    \begin{align}
    2^{|G|}-B(G)\ &=\ \ 2^{|G|} - \# \{ A \subseteq G:d(A)=0\} \notag
    \\ & \leq\ \ 2^{|G|} -\# \{ A: AA=AA^{-1}=G\} \notag
    \\ &=\ \ \# \{ A: \exists g \not \in AA \text{ or } \exists h \not \in AA^{-1}\} \notag
    \\ &\leq\ \ \sum_{g \in G}\# \{ A: g \not \in AA \text{ or } g \not \in AA^{-1}\} \notag
    \\ &=\ \ \sum _{g \in G}\paren{i(\mc F_{P}(g)) + i(\mc F_{Q}(g)) - i (\mc F_{PQ}(g,g))} \notag
    \\ &\leq\ \ \sum_{g \in G}\paren{i(\mc F_P(g)) + i (\mc F_{Q}(g))}.
\end{align}
\end{proof}

We reformulate the desired inequality in \eqref{eq:dom_goal} as a statement about $P_k,Q_k,PP_k$, and $QQ_k$.

\begin{lem} \label{lem:uv_criterion}
    Let $G$ be a finite group, and let $P,Q,PP,QQ$ be as in \Cref{lem:PQPPQQ}, formed with $G$ in place of $G_k$. Suppose $P \neq Q$, and set
    \begin{align}
        X\ :=\ \max\{P,Q\},\quad u\ :=\ \frac{\min\{P,Q\}}{X}, \quad v\ :=\ \frac{PP+QQ}{X}.
    \end{align}
    If
    \begin{align}
        f(u,v)\ :=\ u^2-6u+1-v(1+u)\ >\ 0, \label{eq:f_criterion}
    \end{align}
    then $G$ is product-dominant when $Q>P$, and quotient-dominant when $P>Q$. Moreover, the following hold.
    \begin{enumerate}[(i)]
        \item The condition \eqref{eq:f_criterion} is equivalent to
        \begin{align}
            2(Q-P)^2\ >\ (Q+P)\paren{Q+P+PP+QQ}. \label{eq:f_product_form}
        \end{align}
        \item The function $f$ is strictly decreasing in $u$ on $[0,3)$ and in $v$ on $[0,\infty)$. Hence, if $u \leq u^+ < 3$ and $v \leq v^+$, then $f(u,v) \geq f(u^+,v^+)$, so it suffices to verify \eqref{eq:f_criterion} for any such upper bounds $u^+$ and $v^+$.
        \item The condition \eqref{eq:f_criterion} holds whenever $u \leq 1$ and $6u+2v<1$.
    \end{enumerate}
\end{lem}

\begin{proof}
    By \eqref{eq:M1_in_terms}, we have 
    \begin{align}
        |M_1(G)| &= |Q-P| = X(1-u),
        \\ 2^{|G|}-B(G)\ &\leq\ Q+P\ =\ X(1+u),
        \\ M_2(G)\ &\leq\ Q+P+QQ+PP\ =\ X(1+u+v).
    \end{align}
Thus,
    \begin{align}
        \frac{(M_1(G))^2}{(2^{|G|}-B(G))\cdot M_2(G)}\ \geq\ \frac{X^2(1-u)^2}{X(1+u)\cdot X(1+u+v)}\ =\ \frac{(1-u)^2}{(1+u)(1+u+v)}. \label{eq:uv_ratio}
    \end{align}
    Since $(1+u)(1+u+v)>0$, we have
    \begin{align}
        \frac{(1-u)^2}{(1+u)(1+u+v)}\ >\ \frac 12\ &\iff\ 2(1-u)^2\ >\ (1+u)(1+u+v) \notag
        \\ &\iff\ u^2-6u+1-v(1+u)\ >\ 0,
    \end{align}
    so \eqref{eq:f_criterion} gives \eqref{eq:dom_ineq}. The conclusion now follows from \Cref{thm:dom_ineq}, as $M_1(G)=Q-P$ is positive exactly when $Q>P$.

    For (i), multiply \eqref{eq:f_criterion} through by $X^2>0$ and use $X(1-u) = |Q-P|$, $X(1+u) = Q+P$, and $Xv = PP+QQ$, which turns $2(1-u)^2 > (1+u)(1+u+v)$ into \eqref{eq:f_product_form}. For (ii), let $u \in [0,3)$ and $v \in [0,\infty)$, so
    \begin{align}
        \frac{\partial f}{\partial u}\ =\ 2u-6-v\ <\ 0, \qquad  \frac{\partial f}{\partial v}\ =\ -(1+u)\ <\ 0.
    \end{align}
    For (iii), note that $u^2 \geq 0$ and $1+u \leq 2$, so $f(u,v) \geq 1-6u-2v$.
\end{proof}

We next collect the bounds from \Cref{sec:M1,sec:M2} in a common form. Let $G$ be a finite group of order $N$, let $T$, $Z$, and $R$ be as in \Cref{thm:mptqresult} (formed with $G$ in place of $G_k$), and let $P,Q,PP,QQ$ be as in \Cref{lem:PQPPQQ}. Since $7^{1/4}<1.68$, \Cref{cor:FQ_asym} and \Cref{lem:FP_asym} give
\begin{align}
    T\sqrt 3^{N}\ &\leq\ Q\ \leq\ T\sqrt 3^N + N\cdot 1.68^N, \label{eq:Q_explicit}
    \\ Z \sqrt 3^{N-R/Z}\ &\leq\ P\ \leq\ Z\sqrt 3^N + N\cdot 1.68^N. \label{eq:P_explicit}
\end{align}
Since the sums defining $PP$ and $QQ$ run over the $N(N-1)$ ordered pairs $(g,h)$ with $g \neq h$, \Cref{lem:double_bound} and $19^{1/6}<1.68$ give
\begin{align}
    PP+QQ\ \leq\ 2N(N-1)\cdot 1.68^N\ \leq\ 2N^2\cdot 1.68^N. \label{eq:double_explicit}
\end{align}
Finally, as $1.68<\sqrt 3$, we remark that the quantity
\begin{align}
    \veps(N)\ :=\ N^3\paren{\frac{1.68}{\sqrt 3}}^N \label{eq:eps_def}
\end{align}
satisfies $\veps(N)\to 0$ as $N \to \infty$. The error terms obtained below are controlled by this $\veps(N)$.

To prove statement (i) of the theorem, suppose $Z_k = o(T_k)$. As $Z_k \geq 1$ for every $k$, this forces $T_k \to \infty$; in particular $T_k \geq 1$ for all large $k$, so $Q_k \geq T_k\sqrt 3^{N_k}>0$ by \eqref{eq:Q_explicit}. Dividing the upper bound in \eqref{eq:P_explicit} by $T_k \sqrt 3^{N_k}$ gives
\begin{align}
    \frac{P_k}{Q_k}\ \leq\ \frac{Z_k\sqrt 3^{N_k}+N_k\cdot 1.68^{N_k}}{T_k\sqrt 3^{N_k}}\ \leq\ \frac{Z_k}{T_k}+N_k\paren{\frac{1.68}{\sqrt 3}}^{N_k}\ \leq\ \frac{Z_k}{T_k}+\veps(N_k), \label{eq:case1_u}
\end{align}
and \eqref{eq:double_explicit} gives
\begin{align}
    \frac{PP_k+QQ_k}{Q_k}\ \leq\ \frac{2N_k^2\cdot 1.68^{N_k}}{T_k\sqrt 3^{N_k}}\ \leq\ 2\veps(N_k). \label{eq:case1_v}
\end{align}
The right-hand sides of \eqref{eq:case1_u} and \eqref{eq:case1_v} tend to $0$, since $Z_k/T_k \to 0$ by assumption and $\veps(N_k)\to 0$. Hence, for all sufficiently large $k$, we have $P_k<Q_k$ and
\begin{align}
    6\cdot \frac{P_k}{Q_k} + 2\cdot \frac{PP_k+QQ_k}{Q_k}\ <\ 1 .
\end{align}
Applying \Cref{lem:uv_criterion} with $X=Q_k$, $u = P_k/Q_k$, and $v = (PP_k+QQ_k)/Q_k$, we conclude that $G_k$ is product-dominant for all sufficiently large $k$.

Now, for case (ii) of the theorem, suppose $T_k\cdot \sqrt 3^{R_k/Z_k} = o(Z_k)$, and set
\begin{align}
    W_k\ :=\ Z_k \sqrt 3^{N_k-R_k/Z_k}\ \leq\ P_k, \label{eq:Wk_def}
\end{align}
the lower bound from \eqref{eq:P_explicit}. We first check that
\begin{align}
    \sqrt 3^{R_k/Z_k}\ \leq\ N_k \label{eq:sqrt3_le_N}
\end{align}
for all sufficiently large $k$. If $T_k \geq 1$, then 
\begin{align}
    \sqrt 3^{R_k/Z_k}\  \leq\  T_k \sqrt 3^{R_k/Z_k} \ <\  Z_k \ \leq\ N_k
\end{align}
for large $k$, by the assumption $T_k\sqrt 3^{R_k/Z_k}=o(Z_k)$. If instead $T_k=0$, then $G_k$ has no element of order $2$, and therefore has odd order by Cauchy's theorem. In that case, if $x^2 \in Z(G_k)$, then $\bar x^2 = \bar e$ in the quotient group $G_k / Z(G_k)$, which also has odd order, so $\bar x = \bar e$ and $x \in Z(G_k)$. Thus $R_k = Z_k$ and $\sqrt 3^{R_k/Z_k} = \sqrt 3 \leq N_k$. This proves \eqref{eq:sqrt3_le_N}, and combining it with \eqref{eq:Wk_def} yields
\begin{align}
    W_k\ \geq\ \frac{Z_k\sqrt 3^{N_k}}{N_k}\ \geq\ \frac{\sqrt 3^{N_k}}{N_k}. \label{eq:Wk_lower}
\end{align}
Using \eqref{eq:Q_explicit}, \eqref{eq:Wk_def}, and \eqref{eq:Wk_lower}, we obtain
\begin{align}
    \frac{Q_k}{P_k}\ \leq\ \frac{T_k\sqrt 3^{N_k}+N_k\cdot 1.68^{N_k}}{W_k}\ \leq\ \frac{T_k\cdot \sqrt 3^{R_k/Z_k}}{Z_k} + N_k^2\paren{\frac{1.68}{\sqrt 3}}^{N_k}\ \leq\ \frac{T_k\cdot \sqrt 3^{R_k/Z_k}}{Z_k}+\veps(N_k), \label{eq:case2_u}
\end{align}
where the middle inequality uses $T_k\sqrt 3^{N_k}/W_k = T_k\sqrt 3^{R_k/Z_k}/Z_k$. Similarly, \eqref{eq:double_explicit} and \eqref{eq:Wk_lower} give
\begin{align}
    \frac{PP_k+QQ_k}{P_k}\ \leq\ \frac{2N_k^2\cdot 1.68^{N_k}}{W_k}\ \leq\ 2N_k^3\paren{\frac{1.68}{\sqrt 3}}^{N_k}\ =\ 2\veps(N_k). \label{eq:case2_v}
\end{align}
The right-hand sides of \eqref{eq:case2_u} and \eqref{eq:case2_v} tend to $0$, since $T_k \sqrt 3^{R_k/Z_k}/Z_k \to 0$ by assumption and $\veps(N_k) \to 0$. Hence, for all sufficiently large $k$, we have $Q_k<P_k$ and
\begin{align}
    6\cdot \frac{Q_k}{P_k} + 2\cdot \frac{PP_k+QQ_k}{P_k}\ <\ 1.
\end{align}
Applying \Cref{lem:uv_criterion} with $X=P_k$, $u = Q_k/P_k$, and $v = (PP_k+QQ_k)/P_k$, we conclude that $G_k$ is quotient-dominant for all sufficiently large $k$.
\end{proof}

%%%%%%%%%%%%%%%%%%%%%%%%%%%%%%%%%%%%%%%%%%%%%%%%%%%%%%%%%%%%%%%%%%%%%%
\section{\texorpdfstring{Proof of \Cref{conj:miller_vissuet} and other applications}{Proof of Miller--Vissuet Conjecture and other applications}} \label{sec:proofandcors}

We now apply \Cref{thm:mptqresult} to several families of groups. For dihedral and abelian groups, we also make the bounds of \Cref{sec:dom_thm} explicit, which lets us handle all group orders rather than only sufficiently large ones.

%%%%%%%%%%%%%%%%%%%%%%%%%%%%%%%%%%%%%%%%%%%%%%%%%%%%%%%%%%%%%%%%%%%%%%
\subsection{Product-dominant groups}

We begin by resolving \Cref{conj:miller_vissuet}. Let $D_k$ be the dihedral group of order $N=2k$. Proving that $D_k$ is product-dominant for sufficiently large $k$ is quick: with $Z_k,T_k$ as in \Cref{thm:mptqresult}, we have
\begin{align}
    Z_k \ &= \ \case{2 & k \text{ even}, \\ 1 & k \text{ odd},}
    \\ T_k \ &= \ \case{k+1 & k \text{ even,} \\ k & k\text{ odd,}}
\end{align}
so $Z_k = o(T_k)$ and \Cref{thm:mptqresult} gives the desired result. 

Now, we compute the threshold which we have been calling ``sufficiently large'', so that we can manually validate the conjecture for all smaller dihedral groups\footnote{From the perspective of \Cref{thm:mptqresult}, there's nothing particular special about dihedral groups. We apply this explicit threshold analysis as a proof-of-concept for our methodology, and to fully resolve \Cref{conj:miller_vissuet}. Similar analysis can be applied to other families of groups, but for brevity we settle for ``sufficiently large''.}.

For dihedral groups, the general bound \eqref{eq:FPP_bound} is too weak to give a small explicit threshold (small enough for computation), so we first sharpen it.

\begin{lem} \label{lem:dihedral_PP}
    Let $D_k$ be the dihedral group of order $N=2k$, with $k \geq 3$. Then for all $g \neq h \in D_k$,
    \begin{align}
        i(\mc F_{PP}(g,h))\ \leq\ \frac 97\cdot 19^{N/6}. \label{eq:dihedral_PP}
    \end{align}
\end{lem}

\begin{proof}
    Write $\rho$ for a rotation of order $k$, so the rotations form the cyclic subgroup $\angles{\rho}$ and every other element is a reflection. First suppose $g$ or $h$ is a reflection, say $g$. The centralizer of a reflection has order at most $4$, so, as in the proof of \Cref{lem:FP_asym}, at most $4$ vertices lie in $2$-cycles of $\mc F_P(g)$ and
    \begin{align}
        i(\mc F_{PP}(g,h))\ \leq\ i(\mc F_P(g))\ \leq\ \sqrt 3^{\,4}\cdot 7^{(N-4)/4}\ =\ \frac 97 \cdot 7^{N/4}\ \leq\ \frac 97\cdot 19^{N/6}. \label{eq:PP97dihedral}
    \end{align}
    Now suppose $g = \rho^a$ and $h = \rho^b$ are both rotations, with $a \neq b$ modulo $k$. A product $xy$ is a rotation exactly when $x$ and $y$ are both rotations or both reflections, so $\mc F_{PP}(g,h)$ has no edges between rotations and reflections, and $i(\mc F_{PP}(g,h))$ is the product of the numbers of independent sets of the two induced subgraphs. The subgroup of rotations is abelian, so we are in the situation of the last part of the proof of \Cref{lem:double_bound} with $G$ replaced by $\angles{\rho}$, hence the rotations contribute at most $7^{k/4}$.
    
    On the reflections, a reflection $x$ is joined to $xg,\ gx,\ xh,\ hx$, which are the reflections obtained from $x$ by rotating through $\pm a$ and $\pm b$. If $a = 0$ or $b=0$, every reflection is self-adjacent and contributes a factor of $1$. Otherwise, the reflections form a simple graph with every degree in $\{2,3,4\}$, so they contribute at most $19^{k/6}$ by \Cref{cor:Zhao_strong_indep}. Since $7^{1/4} < 19^{1/6}$, we conclude that
    \begin{align}
        i(\mc F_{PP}(g,h))\ \leq\ 7^{k/4}\cdot 19^{k/6}\ \leq\ 19^{k/3}\ =\ 19^{N/6}.
    \end{align}
    As $19^{N/6} \leq \frac 97 \cdot 19^{N/6}$, the bound in \eqref{eq:PP97dihedral} is the larger of the two, so \eqref{eq:dihedral_PP} holds in both cases.
\end{proof}

\begin{thm} \label{thm:dihedralgroups}
    All dihedral groups $D_n$ for $n \geq 3$ are product-dominant.
\end{thm}

\begin{proof}
    Let $D_k$ be the dihedral group of order $N_k := 2k$. We have
    \begin{align}
        Z_k\ =\ \case{2 & k \text{ even}, \\ 1 & k \text{ odd},} 
    \end{align}
    and
    \begin{align}
        T_k\ =\ \case{k+1 & k \text{ even,} \\ k & k\text{ odd.}}
    \end{align}
    Let $Q_k,P_k,QQ_k,PP_k$ be as in \Cref{lem:PQPPQQ}. As $7^{1/4}<1.68$, \Cref{cor:FQ_asym} and \Cref{lem:FP_asym} give
    \begin{align}
        T_k \cdot \sqrt 3^{N_k}\ &\leq\ Q_k\ \leq\ T_k \cdot \sqrt 3^{N_k} + (N_k-T_k)\cdot 1.68^{N_k},
        \\ P_k\ &\leq\ Z_k \cdot \sqrt 3^{N_k} + (N_k-Z_k)\cdot 1.68^{N_k}.
    \end{align}
    Furthermore, since the sums defining $QQ_k$ and $PP_k$ run over the $N_k(N_k-1)$ ordered pairs $(g,h)$ with $g \neq h$, \Cref{lem:double_bound,lem:dihedral_PP} give
    \begin{align}
        QQ_k \ &\leq\  N_k(N_k-1) \cdot 19^{N_k/6},
        \\ PP_k\ &\leq\ \frac 97\cdot N_k(N_k-1) \cdot 19^{N_k/6}.
    \end{align}
Let
\begin{align}
    U_k\ :=\ (T_k + Z_k ) \sqrt 3^{N_k}+(2N_k-T_k-Z_k)\cdot 1.68^{N_k}. \label{eq:dihedral_Uk}
\end{align}
Then $2^{N_k}-B(G_k) \leq Q_k+P_k \leq U_k$ and $M_2(G_k) \leq U_k + \frac{16}{7}N_k(N_k-1)\cdot 19^{N_k/6}$. Hence, by \Cref{thm:dom_ineq} and \Cref{lem:PQPPQQ}, we want to find $k_0$ such that $k \geq k_0$ implies
\begin{align}
    \frac{\braks{(T_k - Z_k ) \sqrt 3^{N_k}-(N_k-Z_k)\cdot 1.68^{N_k}}^2}{U_k\paren{U_k+\frac{16}{7}N_k(N_k-1)\cdot 19^{N_k/6}}}\ >\ \frac 12. \label{eq:dihedral_ineq}
\end{align}
Recall that $N_k=2k$. When $k$ is even, the left-hand side of \eqref{eq:dihedral_ineq} evaluates to
\begin{align}
    \frac{\braks{ (k-1)3^k-(2k-2)1.68^{2k} }^2}{\braks{ (k+3)3^k+(3k-3)1.68^{2k}}\braks{(k+3)3^k+(3k-3)1.68^{2k} + \frac{32}{7}k(2k-1)19^{k/3}}}.
\end{align}
One can check that this expression is strictly larger than $1/2$ for all even $k \geq 62$. When $k$ is odd, the left-hand side of \eqref{eq:dihedral_ineq} evaluates to
\begin{align}
    \frac{\braks{ (k-1)3^k-(2k-1)1.68^{2k} }^2}{\braks{ (k+1)3^k+(3k-1)1.68^{2k}}\braks{(k+1)3^k+(3k-1)1.68^{2k} + \frac{32}{7}k(2k-1)19^{k/3}}}.
\end{align}
One can check that this expression is strictly larger than $1/2$ for all odd $k \geq 61$. Thus, $D_k$ is product-dominant for all $k\geq 61$, and we are left to check all $k\in \{3,\dots,60\}$ manually. The details of this computation are given in \Cref{appendix:computation}.
\end{proof}

\begin{thm} \label{thm:symmetricgroups}
    All sufficiently large symmetric groups are product-dominant.
\end{thm}

\begin{proof}
    For symmetric groups $\{S_k\}$ with $k\geq3$, we have $Z_k=1$ and $T_k \geq \binom{k}{2}$, so $Z_k = o(T_k)$, and case (i) of \Cref{thm:mptqresult} applies. 
\end{proof}

\begin{thm} \label{thm:simplegroups}
    All sufficiently large non-abelian finite simple groups are product-dominant.
\end{thm}

\begin{proof}
Let $G$ be a non-abelian finite simple group. Then $Z = |Z(G)| = 1$, since $Z(G)$ is a normal subgroup of $G$ and is proper because $G$ is non-abelian. Moreover, by the Feit--Thompson theorem \cite{feit-thompson}, $G$ has even order, hence has an element $g$ of order $2$. Let $\mathcal C_g$ be the conjugacy class of $g$. Then, the action of $G$ on $\mathcal C_g$ by conjugation induces a homomorphism $\varphi :G \to S(\mathcal C_g)$, where $S(\mathcal C_g)$ denotes the symmetric group on $\mathcal C_g$. If $\ker \varphi = G$, then $xgx^{-1}=g$ for all $x \in G$, hence $g \in Z(G)$ which cannot be the case. Therefore, $\ker\varphi$ is a proper normal subgroup of $G$, thus is trivial. So, $G$ embeds in $S(\mathcal C_g)$, and $|G| \leq |\mathcal C_g|!$. Therefore, if $\{G_k\}$ is a sequence of non-abelian finite simple groups with $|G_k| \to \infty$, we must have an involution $g_k \in G_k$ for all $k$, and since $|G_k| \leq |\mathcal C_{g_k}|!$, we must have $|\mathcal C_{g_k}| \to \infty$. Finally, every element in the conjugacy class $\mathcal C_{g_k}$ must also have order $2$, thus $T_k \geq |\mathcal C_{g_k}| \to \infty$, so it follows that $Z_k = 1 = o(T_k)$, and case (i) of \Cref{thm:mptqresult} applies.
\end{proof}

%%%%%%%%%%%%%%%%%%%%%%%%%%%%%%%%%%%%%%%%%%%%%%%%%%%%%%%%%%%%%%%%%%%%%%
\subsection{Quotient-dominant groups}
\begin{thm} \label{thm:oddgroups}
    All sufficiently large groups of odd order are quotient-dominant.
\end{thm}

\begin{proof}
    Let $\{G_k\}$ be a sequence of groups of odd order. Then, $T_k=0$ and $Z_k \geq 1$, so we immediately have
    \begin{align}
        T_k \cdot \sqrt 3^{R_k/Z_k}\ =\ 0\ =\ o(Z_k),
    \end{align}
    and case (ii) of \Cref{thm:mptqresult} applies. 
\end{proof}

For finite abelian groups, due to their prevalence in relation to the classic MSTD problem (see \cite{nathanson_history}), we go deeper to obtain more concrete results. In terms of the orders of their elements, abelian groups can vary greatly. For example, \cite{noMSTDpenmanwells} showed that abelian groups of the form $C_2^k$ are all balanced, but most other abelian groups are quotient-dominant, and in general abelian groups with many involutions are harder to analyze, as $Z_k$ and $T_k$ from \Cref{thm:mptqresult} are then of similar size. Hence, a general classifying theorem is harder to obtain for abelian groups, but the following result gets closer to classification with some extra strategies.

\begin{thm} \label{thm:abeliangroups}
    Let $G$ be a finite abelian group, written additively, let $2G := \{2g : g \in G\}$, and let $m:=|2G|$. Suppose $m\geq 6$ and $|G| \geq 171$. Then, $G$ is quotient-dominant.
\end{thm}
\vspace{-10pt}
\begin{proof}
    Let $N := |G|$, and let $T$ be the number of elements of order $2$. The kernel of the homomorphism $x\mapsto 2x$ consists of the identity and the $T$ elements of order $2$, so the first isomorphism theorem gives $N = m\cdot (T+1)$.

    Let $P$ and $Q$ be as in \Cref{lem:PQPPQQ} (formed with $G$ instead of $G_k$). Since $G$ is abelian, $Z(G)=G$, and the number $r(g)$ of square roots of $g$ is $T+1$ if $g \in 2G$ and $0$ otherwise. Hence, by \eqref{eq:FP_true_sum} and \Cref{cor:FQ_asym},
    \begin{align}
        P \ &=\ (N-m)\sqrt{3}^N + m\sqrt{3}^{N-(T+1)}  \ =\ mT\sqrt{3}^N + m\sqrt{3}^{N-(T+1)},  \label{eq:abelianP} \\
        Q \ &\leq\ T\sqrt{3}^N + (N-T)7^{N/4}. \label{eq:abelianQ}
    \end{align}
    Let $Q^+ := T\sqrt{3}^N + (N-T)7^{N/4}$, so $Q \leq Q^+$. Since $m\geq 6$ and $5T + 6\cdot 3^{-(T+1)/2} \geq 3$ for every $T \geq 0$, \eqref{eq:abelianP} and \eqref{eq:abelianQ} give
    \begin{align}
        M_1(G)\ \overset{\eqref{eq:M1_sum}}{=}\ Q - P \ \leq\ Q^+-P \ \leq\ \sqrt 3^N\paren{N(7/9)^{N/4} - 3}\ <\ 0,
    \end{align}
    where the last inequality holds as $N \geq 171$. Hence case (ii) of \Cref{thm:dom_ineq} applies. Since $G$ is abelian, every pair $g,h$ is central, so \eqref{eq:FPP_central} gives $PP \leq N(N-1)\cdot 7^{N/4}$, while \eqref{eq:FQQ_bound} gives $QQ \leq N(N-1)\cdot 19^{N/6}$ (here $PP$ and $QQ$ are as in \Cref{lem:PQPPQQ}, not to be confused with the products $P\cdot P$ and $Q\cdot Q$). Write $W := N(N-1)\paren{19^{N/6}+7^{N/4}}$, so that $PP+QQ \leq W$.

    As $Q<P$, we may apply \Cref{lem:uv_criterion} with $X=P$, and the quantities $u,v$ of that lemma satisfy
    \begin{align}
        \frac QP\ \leq\ \frac{Q^+}{P}\ =:\ u \quad \quad \text{and} \quad \quad \frac{PP+QQ}{P}\ \leq\ \frac WP\ =:\ v.
    \end{align}
    By \Cref{lem:uv_criterion}(ii), it therefore suffices to prove that
    \begin{align}
        f(u,v)\ =\ u^2-6u+1-v(1+u)\ >\ 0 \label{eq:abelian_f_goal}
    \end{align}
    for these $u$ and $v$. We find explicit upper bounds for $u$ and $v$ which depend only on $N$. If $T=0$, then $m=N$ and $P = N\sqrt 3^{N-1}$, so $u = \sqrt 3\,(7/9)^{N/4}$ and $v = \sqrt 3\,(N-1)\paren{(19/27)^{N/6}+(7/9)^{N/4}}$, and both bounds below hold. If $T \geq 1$, we have
    \begin{align}
        u \ &=\ \frac{T\sqrt{3}^N+(N-T)7^{N/4}}{(N-m)\sqrt{3}^N + m\sqrt{3}^{N-(T+1)}} \notag \\
        &\leq\ \frac{T\sqrt{3}^N+(N-T)7^{N/4}}{mT\sqrt{3}^N} \notag \\
        &=\ \frac{1}{m} + \frac{N-T}{mT} (7/9)^{N/4} \notag \\
        &\leq\ \frac{1}{6} + 2\left(\frac{7}{9}\right)^{N/4}
    \end{align}
    as $(N-T)/(mT) < N/(mT) = (T+1)/T \leq 2$. Moreover,
    \begin{align}
        v\ &=\ \frac{N(N-1)\paren{19^{N/6}+7^{N/4}}}{(N-m)\sqrt{3}^N+m\sqrt{3}^{N-(T+1)}} \notag \\
        &\leq\ \frac{N(N-1)\paren{19^{N/6}+7^{N/4}}}{mT\sqrt{3}^N} \notag \\
        &\leq\ 2(N-1)\paren{\left(\frac{19}{27}\right)^{N/6}+\left(\frac{7}{9}\right)^{N/4}}, 
    \end{align}
    since $N/(mT) = (T+1)/T \leq 2$. Thus, 
    \begin{align}
        f(u,v)\ &\geq\ f\left(\frac{1}{6}+2\left(\frac{7}{9}\right)^{N/4},\ 2(N-1)\paren{\left(\frac{19}{27}\right)^{N/6}+\left(\frac{7}{9}\right)^{N/4}}\right) \ =:\ h(N).
    \end{align}
    Finally, it is not hard to show that $h(N)$ is increasing for every $N\geq 12$, and numerical analysis shows that $h(N) > 0$ for all $N\geq 171$. 
    \end{proof}

\begin{rmk} \label{rmk:abelianthreshold}
    The threshold $N=171$ can be lowered computationally in the same way as for dihedral groups, as described in \Cref{appendix:computation}. However, due to the size of these groups, as well as the quantity with which they occur, it's unlikely that the present methods would resolve the rest of the abelian groups with $|2G| \geq 6$ and $|G| \leq 170$
\end{rmk}

%%%%%%%%%%%%%%%%%%%%%%%%%%%%%%%%%%%%%%%%%%%%%%%%%%%%%%%%%%%%%%%%%%%%%%
\section{\texorpdfstring{Computing $M_1(G)$ by characters and conjugacy classes}{Computing M1(G) by characters and conjugacy classes}} \label{sec:proving_expthm}

We aim to find the length of the cycle containing a vertex $x$ in $\mc F_P(g)$ in order to compute 
\begin{align}
    \sum_{g\in G}i(\mc F_P(g)).
\end{align}

\begin{defn}
    Recall from \Cref{sec:M1} the definition of $p_g^{(n)}(x)$. For $g,x \in G$, let $\ell_{g,x}$ denote the minimal positive integer $\ell$ such that $p_g^{(\ell)}(x)=x$. Thus, the cycle containing $x$ in $\mc F_P(g)$ has length $\ell_{g,x}$.
\end{defn}

\begin{lem} \label{lem:divisibility}
    If $p^{(d)}_g(x)=x$, then $\ell_{g,x}\mid d$.
\end{lem}
\begin{proof}
Write $\ell := \ell_{g,x}$ and $d = q\ell + s$ with $q \geq 0$ and $0 \leq s < \ell$. Since $p_g^{(\ell)}(x)=x$, applying this $q$ times gives $p_g^{(q\ell)}(x) = x$, hence
\begin{align}
    x\ =\ p_g^{(d)}(x)\ =\ p_g^{(s)}\paren{p_g^{(q\ell)}(x)}\ =\ p_g^{(s)}(x), \label{eq:div_alg}
\end{align}
where $p_g^{(0)}$ denotes the identity map. By the minimality of $\ell$, we must have $s=0$, so $\ell \mid d$.
\end{proof}

\begin{lem} \label{lem:poss_lengths}
    For any $g,x \in G$, $\ell_{g,x}$ divides $2\ord(g)$.
\end{lem}

\begin{proof}
    By \Cref{lem:pgqg_forms},
    \begin{align}
        p^{(2\ord(g))}_g(x)\ =\ g^{-\ord(g)}xg^{\ord(g)}\ =\ x,
    \end{align}
    so by \Cref{lem:divisibility} we have $\ell_{g,x} \mid 2\ord(g)$.
\end{proof}

To determine the minimal $\ell$ with $p_g^{(\ell)}(x)=x$, it helps to know how many $x \in G$ satisfy $p_g^{(d)}(x)=x$ for a fixed positive integer $d$. This can be computed with character theory, as the irreducible characters of a group provide counting formulas for the size of the centralizer of an element, and the number of square roots of an element.

\begin{lem} \label{lem:Ldg_formula}
    For any $g \in G$, $d \in \Z^+$, let
    \begin{align}
        L(d,g)\ :=\ \#\brax{ x\in G: p_g^{(d)}(x)=x}.
    \end{align}
    Let $\chi_1,\dots, \chi_n$ be the irreducible characters of $G$, and let $\nu$ be the Frobenius--Schur indicator
    \begin{align}
        \nu(\chi)\ :=\ \frac{1}{|G|}\sum_{h \in G}\chi(h^2).
    \end{align}
    If $d=2k$ is even, then
    \begin{align}
        L(2k,g)\ =\ \abs{C_G(g^k)}\ =\ \sum_{i=1}^n\abs{\chi_i(g^k)}^2. \label{eq:L_even_d}
    \end{align}
    If $d=2k+1$ is odd, then
    \begin{align} \label{eq:L_odd_d}
        L(2k+1,g)\ =\ \sum_{i=1}^n \nu(\chi_i)\cdot \chi_i(g^{-2k-1}).
    \end{align}
\end{lem}

\begin{proof}
    Recall that $p_g^{(2k)}(x) = x \iff xg^k=g^kx$. Thus, $L(2k,g) = |C_G(g^k)|$. The second equality in \eqref{eq:L_even_d} follows from the second orthogonality relation for characters.

    When $d=2k+1$ is odd, \Cref{lem:pgqg_forms} gives
    \begin{align}
        p_g^{(2k+1)}(x)\ =\ x \ \iff\ g^{-k}x^{-1}g^{k+1}\ =\ x.
    \end{align}
    Substituting $x=g^{k+1}y$ yields
    \begin{align}
        p_g^{(d)}(x)=x\ &\iff\  g^{-k}y^{-1}g^{-(k+1)}g^{k+1}=g^{k+1}y \notag
        \\ &\iff\ g^{-d}=y^2.
    \end{align}
Our substitution induces a bijection on $G$, hence the number of $y \in G$ satisfying $g^{-d}=y^2$ is precisely $L(d,g)$. The number of square roots of $g^{-d}$ is known in general to be
\begin{align}
    \sum_{i=1}^n\nu(\chi_i) \cdot \chi_i(g^{-d})
\end{align}
(see \cite{isaacs2006character} or \cite{rowe2022squareroots}, for example). This proves \eqref{eq:L_odd_d}.
\end{proof}

Having counted the $x$ satisfying $p_g^{(d)}(x)=x$, we apply the M\"obius inversion formula (see \cite{MobiusInversion}) to count the $x$ for which $d$ is the minimal positive integer with $p_g^{(d)}(x)=x$.

\begin{thm}\label{thm:num_of_cycles_prod}
    Fix $g\in G$. Let $c(m,g)$ be the number of $m$-length cycles in $\mc F_P(g)$. Then
    \begin{align}
        c(m,g)\ =\ \frac{1}{m}\sum_{d\mid m}\mu\left(\frac{m}{d}\right)L(d,g).
    \end{align}
\end{thm}
\begin{proof}
    Define
    \begin{align}
    f(d) \ :=\ \#\{x\in G:\ell_{g,x}=d\}.
    \end{align} 
   Then, by \Cref{lem:divisibility}, we see that
    \begin{align}
        L(m,g)\ =\ \sum_{d\mid m}f(d).
    \end{align}
    Thus, by the Möbius inversion formula, we have 
    \begin{align}
        f(m)\ =\ \sum_{d\mid m}\mu(m/d)\cdot L(d,g), \label{eq:mobiususe}
    \end{align}
    where $\mu$ is the Möbius function.
    This gives us the number of $x\in G$ contained in $m$-cycles in $\mc F_P(g)$. We divide by $m$ to obtain the total number of $m$-length cycles.
\end{proof}

\begin{thm} \label{thm:expthm}
    Let $G$ be a finite group, let $\mc C_1,\dots,\mc C_n$ be the distinct conjugacy classes of $G$, and let $g_i$ be a representative for $\mc C_i$. Then
    \begin{align} \label{eq:expecG}
        M_1(G)\ =\ \sum_{i=1}^n|\mc C_i|\paren{ L_{\ord{g_i}}^{|G|/\ord{g_i}} - \prod_{m \mid \paren{2\cdot \ord{g_i}}}L_m^{c(m,g_i)} }.
    \end{align}
\end{thm}

\begin{proof}
By \Cref{lem:M1_sum},
    \begin{align}
        M_1(G)\ =\ \sum_{g \in G}\paren{i(\mc F_Q(g)) - i(\mc F_P(g))}. \label{eq:M1_mention2}
    \end{align}
By the work done in \Cref{sec:M1} (namely \Cref{lem:disjoint_cycles}, \Cref{rmk:lucas_forbiddance},  and \Cref{lem:FQ_form}), as well as the definition of $c(m,g)$ and \Cref{lem:poss_lengths}, we see that
\begin{align}
    i(\mc F_Q(g)) - i(\mc F_P(g))\ = \ L_{\ord(g)}^{|G|/\ord(g)}-\prod_{m \mid (2 \cdot \ord g)}L_m^{c(m,g)} .\label{eq:idiffinlucas}
\end{align}
Moreover, from the computation leading to \eqref{eq:exp_breakdown}, recall that
\begin{align}
    i(\mc F_Q(g)) - i(\mc F_P(g)) \ = \ \#\{A \subseteq G: g\in AA\} - \# \{ A \subseteq G: g \in AA^{-1}\}. \label{eq:subsetcontainment}
\end{align}
Now, for any $g \in G$, $A \subseteq G$, and any automorphism $\sigma$ of $G$, note that
\begin{align}
    g \in AA \ &\iff\ g = xy \text{ for } x,y\in A  \notag
    \\ &\iff\  \sigma(g)=\sigma(x)\sigma(y)  \text{ for } x,y \in A \notag
    \\ &\iff\  \sigma(g) \in \sigma(A)\sigma(A).
\end{align}
We obtain a similar equivalence for $g \in AA^{-1}$. Since $A \mapsto \sigma(A)$ is a bijection on the subsets of $G$, the expression in \eqref{eq:subsetcontainment} takes the same value at $g$ and at $\sigma(g)$. In particular, they are invariant under the action of conjugation, so it follows from \eqref{eq:idiffinlucas} that
\begin{align}
    \sum_{g \in G}\paren{i(\mc F_Q(g)) - i(\mc F_P(g))}\ &=\ \ \sum _{i=1}^n|\mathcal C_i|\paren{i(\mc F_Q(g_i)) - i(\mc F_P(g_i))} \notag
    \\ &=\ \  \sum_{i=1}^n|\mc C_i|\paren{ L_{\ord{g_i}}^{|G|/\ord{g_i}} - \prod_{m \mid \paren{2\cdot \ord{g_i}}}L_m^{c(m,g_i)} }.
\end{align}
\end{proof}

Using \Cref{thm:num_of_cycles_prod,thm:expthm}, we obtain a simpler way to compute $M_1(G)$ for groups whose character theory and conjugacy classes are understood.

%%%%%%%%%%%%%%%%%%%%%%%%%%%%%%%%%%%%%%%%%%%%%%%%%%%%%%%%%%%%%%%%%%%%%%
\section{Future work} \label{sec:futurework}

\Cref{thm:mptqresult} establishes the product/quotient-dominant behavior for many families of groups. However, the limitations of the theorem provide two natural paths for future work.
\begin{enumerate}
    \item For groups that satisfy the conditions of \Cref{thm:mptqresult}, such as those studied in \Cref{sec:proofandcors}, it may be of interest to quantify exactly when the key inequality in \eqref{eq:dom_ineq} is satisfied (as we do for dihedral groups in \Cref{thm:dihedralgroups}). Then, the product/quotient-dominant behavior can be manually computed so long as the remaining groups to check do not get too large.
    \item For groups which do not satisfy the conditions of \Cref{thm:mptqresult}, deeper analysis is required. In particular, these groups may not obtain the dominance of $\sqrt 3^N$ terms in $M_1,M_2,$ or $2^{|G|}-B$. For specific families like dicyclic (or generalized quaternion) groups, it is necessary to analyze the lower order terms to identify when (or if) \eqref{eq:dom_ineq} can be satisfied. For dicyclic groups, Mandal and Neetu \cite{dicycliccounts} recently obtained exact and asymptotic counts of the MSTD, MDTS, and balanced subsets of sizes two and three, which suggests that the dominance behavior of these groups is delicate. For this, formulations like \Cref{thm:num_of_cycles_prod,thm:expthm} could be useful.
\end{enumerate}

On the topic of groups for which the conditions of \Cref{thm:mptqresult} do not apply, we make the following remark.

\begin{rmk}
    There exists a quotient-dominant group with $M_1(G)>0$, namely the binary octahedral group (GAP ID (48,28)). Furthermore, there exists a product-dominant group with $M_1(G)<0$, namely $C_6 \times S_3$ (GAP ID (36,12)). Therefore the method of proving dominance using the inequality in \Cref{thm:dom_ineq} does not work for all groups.
\end{rmk}

It would be interesting to understand what exactly causes counterexamples such as these to occur.

Another natural direction may be to expand on the work of \cite{ZhaoGroups,haviland,generalizeddihedral}, by computing explicit asymptotic formulas on the proportion of subsets which are MPTQ or MQTP. One could also replace $AA$ and $AA^{-1}$ by more general words in $A$ and $A^{-1}$, in the spirit of the generalized MSTD sets of Iyer, Lazarev, Miller, and Zhang \cite{generalizedMSTD}.

Finally, for groups which are particularly resilient to the methodology of \Cref{thm:dom_ineq}, it may be possible to apply more advanced statistical methods by taking higher moments of $d(A)$. These can be understood with forbiddance graphs as follows. For subsets $U,V\subseteq G$, define the generalized forbiddance graph $\mc F(U,V)$ to be the graph where the elements of $U$ are forbidden ``product-wise'' and the elements of $V$ are forbidden ``quotient-wise''. For example, 
\begin{itemize}
    \item $\mc F_P(g) = \mc F(\{g\}, \emptyset)$,
    \item $\mc F_{QQ}(g, h) = \mc F(\emptyset, \{ g,h\}),$
    \item $\mc F_{PQ}(g,h) = \mc F(\{g\}, \{h\})$.
\end{itemize}
Then, one can show similarly to \Cref{lem:2nd_mom} that
\begin{align}
    M_r(G)\ :=\ \sum_{A \subseteq G} d(A)^r\ =\ \sum_{j=0}^r \binom{r}{j}(-1)^{r-j}\sum_{\substack{g_1,\dots,g_{r-j}\in G \\ h_1,\dots,h_j \in G}}i(\mc F(\{g_1,\dots, g_{r-j}\}, \{ h_1,\dots, h_j\})).
\end{align}
Using higher moments, it may be possible to obtain a stronger understanding of the distribution of $d(A)$.

\appendix

\section{Dihedral group computation} \label[appendix]{appendix:computation}

 We explain the computation involved, which can be found in \cite{dihedral_code}. Generally speaking, naively computing whether a group is product- or quotient-dominant requires iterating over all $2^N$ subsets, which becomes very slow, very quickly. However, computing the number of independent sets of a graph can be faster: once a subset of the vertices is not independent, any superset cannot be independent either. Thus, we implement a depth-first search, which starts from the empty set and branches to larger subsets until non-independence is reached, at which point we backtrack. This is especially useful because, by \cite{miller_vissuet}, the proportion of subsets $A$ with $AA = AA^{-1} = G$ converges to $1$; such a subset is independent in no forbiddance graph, so the search prunes early and this algorithm becomes considerably more efficient than the naive one.

 Furthermore, this algorithm can be pushed further if we already know the structure of the graphs we are studying. For instance, we know $\mc F_P$ and $\mc F_Q$ are disjoint unions of cycle graphs, so we need only count the lengths of the component cycles to compute the number of independent sets.

\Cref{thm:dihedralgroups} used \Cref{thm:dom_ineq}, along with our bounds on $M_1,M_2,$ and $2^N-B$, to show that product dominance holds for sufficiently large $k$. Explicit computation of $M_1^2/(M_2\cdot (2^N-B))$ shows that the desired inequality does hold for small $N$, when we use the actual values of $M_1,M_2$, and $B$.
\begin{itemize}
    \item Up to $61$, and far beyond it, computing $M_1$ is no problem, because the algorithm works extremely efficiently on graphs which it knows are disjoint cycle graphs, so \Cref{lem:M1_sum} makes it easy.
    \item Computing $M_2$ is harder, as some of the graphs in \Cref{lem:2nd_mom} can take much longer to compute independent sets for. However, we have the upper bound from \eqref{eq:M2_in_terms} and \eqref{eq:double_explicit}, which reduces computation of $M_2$ to the fast computations for $\mc F_P$ and $\mc F_Q$. 
    \item Computing the number of balanced subsets is the hardest, as it cannot be directly expressed using independent sets of graphs; we only have the inequality \eqref{eq:B_in_terms}, which allows us to bound $2^N-B$ with the fast computations for $\mc F_P$ and $\mc F_Q$.
\end{itemize}
With this in mind, starting from $k=60$ and going down, we use the upper bounds for $M_2$ and $2^N-B$ for as long as they can satisfy \eqref{eq:dom_ineq}. This works down through $k=53$. Then, wanting to avoid having to compute $B$ at all costs, we compute the exact value of $M_2$, continue using the upper bound for $2^N-B$, and this works down through $k=31$ (and technically $k=29$). With nowhere left to turn, we have to compute $B$ explicitly for $k\in \{3,\dots,30\}$. For the larger values of $k$ this required multi-threaded code running on a computing cluster. The desired inequality of \Cref{thm:dom_ineq} indeed held for all $k\in \{3,\dots,60\}$, completing the proof. Note that this need not have been the case, as a group $G$ can be product-dominant without satisfying the inequality \eqref{eq:dom_ineq}.

%%%%%%%%%%%%%%%%%%%%%%%%%%%%%%%%%%%%%%%%%%%%%%%%%%%%%%%%%%%%%%%%%%%%%%
\section*{Acknowledgments}

This research was conducted at the 2026 SMALL REU, with funding from the National Science Foundation (NSF Grant DMS-2341670), Missouri State University, the University of Chicago, the University of Michigan, the University of Rochester, and Williams College. We also thank Atrey Desai for running necessary computations on the UMIACS Nexus Cluster.

%%%%%%%%%%%%%%%%%%%%%%%%%%%%%%%%%%%%%%%%%%%%%%%%%%%%%%%%%%%%%%%%%%%%%%
\subsection*{AI use}
Large language models such as ChatGPT, Claude, and Gemini were used to assist with
\begin{itemize}
    \item finding and helping to resolve two flaws in earlier versions of the proof of \Cref{thm:mptqresult},
    \item refining our proof methods for \Cref{lem:double_bound,lem:uv_criterion,lem:dihedral_PP}, 
    \item searching the literature for applicable known methods, such as Jensen's inequality for \eqref{eq:jensen}, the results in \cite{zhao_strong_indep} for \Cref{cor:Zhao_strong_indep}, and the M\"obius inversion formula for \eqref{eq:mobiususe},
    \item optimizing the code in \cite{dihedral_code}.
\end{itemize}
The authors take full responsibility for the content of this paper.

\printbibliography

\end{document}

%% file: preamble.tex
\usepackage{hyperref, amsmath,amsthm,amsfonts,amssymb,thmtools,mathrsfs,dsfont,faktor,tikz,float,subcaption,quiver,enumerate,mathtools,booktabs,lipsum,xparse,graphicx,epstopdf,comment,indentfirst,bm,yhmath, bbm, thm-restate, setspace, appendix}
\usepackage[dvipsnames]{xcolor}
\usepackage[utf8]{inputenc}

\usepackage[
  backend=biber,
  sorting=nyt, style=alphabetic, maxnames=99
]{biblatex}

\usepackage[most]{tcolorbox}
\usepackage[boxsize=13pt,centertableaux]{ytableau}

\usepackage{fancyhdr}

\usepackage[nameinlink]{cleveref}
\hypersetup{colorlinks=true,linkcolor=blue, citecolor=blue}

\usepackage[margin=1in]{geometry}

\numberwithin{equation}{section}

\usepackage{mathpazo} % Font
\usepackage[skip=3pt plus2pt minus1pt]{parskip} % Paragraph spacing
\graphicspath{{graphics/}}  
\usetikzlibrary{arrows,trees,positioning}

\usepackage{listings}
\definecolor{codegreen}{rgb}{0,0.6,0}
\definecolor{codegray}{rgb}{0.5,0.5,0.5}
\definecolor{codepurple}{rgb}{0.58,0,0.82}
\definecolor{backcolour}{rgb}{0.97,0.97,0.97}
\lstdefinestyle{mystyle}{
    backgroundcolor=\color{backcolour},   
    commentstyle=\color{codegray},
    keywordstyle=\color{blue},
    numberstyle=\tiny\color{codegray},
    basicstyle=\ttfamily,
    breakatwhitespace=true,         
    breaklines=true,                 
    numbers=left,                    
    numbersep=5pt,                  
}
\allowdisplaybreaks % break equations between pages

%% file: thm_setup.tex
\counterwithin{sharedcount}{section}
\declaretheoremstyle[
spaceabove=6pt, 
spacebelow=6pt,
headfont=\normalfont\bfseries,
notefont=\mdseries, notebraces={(}{)},
bodyfont=\normalfont\itshape,
postheadspace=5pt,
headpunct=.
]{regular}

\NewDocumentCommand{\newthmtype}{m m m m m m O{}}{
\declaretheorem[name=#2,
sibling=sharedcount,
refname={#3,#4},
Refname={#5,#6},
style=regular,
#7
]{#1}
\declaretheorem[name=#2,
refname={#3,#4},
Refname={#5,#6},
style=regular,
numbered=no,
#7
]{#1*}
}

\newthmtype{thm}{Theorem}{theorem}{theorems}{Theorem}{Theorems}
\newthmtype{prop}{Proposition}{proposition}{propositions}{Proposition}{Propositions}
\newthmtype{rmk}{Remark}{remark}{remarks}{Remark}{Remarks}
\newthmtype{cor}{Corollary}{corollary}{corollaries}{Corollary}{Corollaries}
\newthmtype{defn}{Definition}{definition}{definitions}{Definition}{Definitions}
\newthmtype{lem}{Lemma}{lemma}{lemmas}{Lemma}{Lemmas}
\newthmtype{ex}{Example}{example}{examples}{Example}{Examples}
\newthmtype{que}{Question}{question}{questions}{Question}{Questions}
\newthmtype{exer}{Exercise}{exercise}{exercises}{Exercise}{Exercises}
\newthmtype{conj}{Conjecture}{conjecture}{conjectures}{Conjecture}{Conjectures}

%% file: references.bib
@article{nathanson_history,
  author        = {Melvyn B. Nathanson},
  title         = {Problems in Additive Number Theory, {I}},
  journal       = {arXiv preprint},
  year          = {2006},
  eprint        = {math/0604340},
  archivePrefix = {arXiv},
  primaryClass  = {math.NT},
  url           = {https://arxiv.org/abs/math/0604340},
}

@incollection{martinobryant,
  title={Many sets have more sums than differences},
  author={Martin, G. and O'Bryant, K.},
  booktitle={Additive Combinatorics},
  series={CRM Proceedings \& Lecture Notes},
  volume={43},
  pages={287--305},
  year={2007},
  publisher={American Mathematical Society},
  address={Providence, RI}
}

@inproceedings{miller_vissuet,
  author    = {Steven J. Miller and Kevin Vissuet},
  editor    = {Melvyn B. Nathanson},
  title     = {Most Subsets Are Balanced in Finite Groups},
  booktitle = {Combinatorial and Additive Number Theory},
  pages     = {147--157},
  publisher = {Springer New York},
  address   = {New York, NY},
  year      = {2014},
  isbn      = {978-1-4939-1601-6},
}

@article{ZhaoGroups,
  author    = {Yufei Zhao},
  title     = {Counting {MSTD} Sets in Finite Abelian Groups},
  journal   = {Journal of Number Theory},
  volume    = {130},
  number    = {10},
  pages     = {2308--2322},
  publisher = {Elsevier BV},
  year      = {2010},
  month     = oct,
  issn      = {0022-314X},
  doi       = {10.1016/j.jnt.2010.06.001},
  url       = {http://dx.doi.org/10.1016/j.jnt.2010.06.001},
}

@article{ZHAO_indep_sets,
  author    = {Yufei Zhao},
  title     = {The Number of Independent Sets in a Regular Graph},
  journal   = {Combinatorics, Probability and Computing},
  volume    = {19},
  number    = {2},
  pages     = {315--320},
  publisher = {Cambridge University Press (CUP)},
  year      = {2009},
  month     = nov,
  issn      = {1469-2163},
  doi       = {10.1017/s0963548309990538},
  url       = {http://dx.doi.org/10.1017/S0963548309990538},
}

@article{zhao_strong_indep,
  author    = {Ashwin Sah and Mehtaab Sawhney and David Stoner and Yufei Zhao},
  title     = {The Number of Independent Sets in an Irregular Graph},
  journal   = {Journal of Combinatorial Theory, Series B},
  volume    = {138},
  pages     = {172--195},
  publisher = {Elsevier BV},
  year      = {2019},
  month     = sep,
  issn      = {0095-8956},
  doi       = {10.1016/j.jctb.2019.01.007},
  url       = {http://dx.doi.org/10.1016/j.jctb.2019.01.007},
}

@misc{statproofbookSecondMoment,
  author       = {Joram Soch},
  title        = {The Second Raw Moment Equals Squared Mean Plus Squared Standard Deviation},
  howpublished = {The Book of Statistical Proofs},
  year         = {2020},
  month        = oct,
  doi          = {10.5281/zenodo.4305949},
  url          = {https://statproofbook.github.io/P/momraw-2nd.html},
}

@article{feit-thompson,
  author    = {Walter Feit and John G. Thompson},
  title     = {Solvability of Groups of Odd Order},
  journal   = {Pacific Journal of Mathematics},
  volume    = {13},
  pages     = {775--1029},
  year      = {1963},
  issn      = {0030-8730},
  url       = {http://projecteuclid.org/euclid.pjm/1103053943},
}

@book{isaacs2006character,
  author    = {Martin Isaacs},
  title     = {Character Theory of Finite Groups},
  publisher = {AMS Chelsea Publishing},
  address   = {Providence, RI},
  year      = {2006},
  isbn      = {978-0821842294},
}

@misc{rowe2022squareroots,
  author       = {Glenn Rowe},
  title        = {Square Roots of Group Elements},
  howpublished = {PhysicsPages},
  year         = {2022},
  month        = aug,
  url          = {https://physicspages.com/pdf/Group%20theory/Square%20roots%20of%20group%20elements.pdf},
}

@misc{MobiusInversion,
  author       = {Eric W. Weisstein},
  title        = {{M\"obius} Inversion Formula},
  howpublished = {MathWorld--A Wolfram Web Resource},
  url          = {https://mathworld.wolfram.com/MoebiusInversionFormula.html}
}

@inproceedings{haviland,
  author       = {Jack Haviland and Ph\'uc L\^am and John Lentfer and Steven J. Miller and Fernando Trejos Su\'arez},
  title        = {More Sums Than Differences Sets in Finite Non-Abelian Groups},
  booktitle    = {Combinatorial and Additive Number Theory {IV}: {CANT} 2019 and 2020},
  pages        = {21--77},
  organization = {Springer},
  year         = {2021},
}

@inproceedings{generalizeddihedral,
  title={Sum and Difference Sets in Generalized Dihedral Groups},
  author={Ascoli, Ruben and Cheigh, Justin and Dantas e Moura, Guilherme Zeus and Jeong, Ryan and Keisling, Andrew and Lilly, Astrid and Miller, Steven J. and Ngamlamai, Prakod and Phang, Matthew},
  booktitle={Combinatorics and Combinatorial Computing MCCCC34, Normal, IL, USA, October 21--23, 2022},
  editor={Bahmanian, Amin},
  series={Springer Proceedings in Mathematics \& Statistics},
  pages={1--30},
  year={2022},
  publisher={Springer}
}

@article{noMSTDpenmanwells,
  author  = {David Penman and Matthew Wells},
  title   = {Sum-Dominant Sets and Restricted-Sum-Dominant Sets in Finite Abelian Groups},
  journal = {Acta Arithmetica},
  volume  = {165},
  pages   = {361--383},
  year    = {2014},
  month   = jan,
  doi     = {10.4064/aa165-4-6},
}

@misc{dihedral_code,
  author = {Yuval Amit},
  title = {MSTD Groups},
  year = {2026},
  publisher = {GitHub},
  journal = {GitHub repository},
  howpublished = {\url{https://github.com/YuvalAmit2/MSTD_Groups.git}}
}

@article{cantelli,
  author  = {B. K. Ghosh},
  title   = {Probability Inequalities Related to {M}arkov's Theorem},
  journal = {The American Statistician},
  volume  = {56},
  number  = {3},
  pages   = {186--190},
  year    = {2002},
  doi     = {10.1198/000313002119},
  url     = {https://www.tandfonline.com/doi/abs/10.1198/000313002119},
}

@article{jensen1906,
  author  = {J. L. W. V. Jensen},
  title   = {Sur les Fonctions Convexes et les In\'egalit\'es entre les Valeurs Moyennes},
  journal = {Acta Mathematica},
  volume  = {30},
  number  = {1},
  pages   = {175--193},
  year    = {1906},
  doi     = {10.1007/BF02418571},
  url     = {https://doi.org/10.1007/BF02418571},
}

@article{ZhaoMissing,
  author  = {Yufei Zhao},
  title   = {Sets Characterized by Missing Sums and Differences},
  journal = {Journal of Number Theory},
  volume  = {131},
  number  = {11},
  pages   = {2107--2134},
  year    = {2011},
  doi     = {10.1016/j.jnt.2011.05.003},
  url     = {https://doi.org/10.1016/j.jnt.2011.05.003},
}

@article{hegartymiller,
  author  = {Peter V. Hegarty and Steven J. Miller},
  title   = {When Almost All Sets Are Difference Dominated},
  journal = {Random Structures \& Algorithms},
  volume  = {35},
  number  = {1},
  pages   = {118--136},
  year    = {2009},
  doi     = {10.1002/rsa.20268},
  url     = {https://doi.org/10.1002/rsa.20268},
}

@article{hemmadylottmiller,
  title={When Almost All Sets Are Difference Dominated in Z/nZ},
  author={Anand Hemmady and Adam Lott and Steven J. Miller},
  journal={Integers},
  year={2017},
  volume={17},
  pages={A54},
  url={https://api.semanticscholar.org/CorpusID:53617459}
}

@article{leftrightquotient,
  title={Comparing Left and Right Quotient Sets in Groups},
  author={Duvivier, Julian and Huang, Xiaoyao and Kennon, Ava and Kwon, Say-Yeon and Miller, Steven J. and Rysmakhanov, Arman and Saldin, Pramana and Watson, Ren},
  journal={Integers},
  volume={26A},
  year={2026}
}

@article{dicycliccounts,
  author        = {Sagar Mandal and Neetu},
  title         = {Exact and Asymptotic Counts of {MSTD}, {MDTS}, and Balanced Sets in Dicyclic Groups},
  journal       = {arXiv preprint},
  year          = {2026},
  eprint        = {2602.09073},
  archivePrefix = {arXiv},
  primaryClass  = {math.NT},
  url           = {https://arxiv.org/abs/2602.09073},
}

@article{nathanson_mstd,
  title={Sets with more sums than differences},
  author={Nathanson, M. B.},
  journal={Integers},
  volume={7},
  pages={A5},
  year={2007}
}

@article{hegarty_explicit,
  author  = {Hegarty, Peter V.},
  title   = {Some explicit constructions of sets with more sums than differences},
  journal = {Acta Arithmetica},
  year    = {2007},
  volume  = {130},
  pages   = {61--77},
  doi     = {10.4064/aa130-1-4}
}

@article{millerorosz,
  author  = {Steven J. Miller and Brooke Orosz and Daniel Scheinerman},
  title   = {Explicit Constructions of Infinite Families of {MSTD} Sets},
  journal = {Journal of Number Theory},
  volume  = {130},
  number  = {5},
  pages   = {1221--1233},
  year    = {2010},
  doi     = {10.1016/j.jnt.2009.09.003},
  url     = {https://doi.org/10.1016/j.jnt.2009.09.003},
}

@article{zhao_ballot,
  author  = {Yufei Zhao},
  title   = {Constructing {MSTD} Sets Using Bidirectional Ballot Sequences},
  journal = {Journal of Number Theory},
  volume  = {130},
  number  = {5},
  pages   = {1212--1220},
  year    = {2010},
  doi     = {10.1016/j.jnt.2009.11.005},
  url     = {https://doi.org/10.1016/j.jnt.2009.11.005},
}

@article{generalizedMSTD,
  author  = {Geoffrey Iyer and Oleg Lazarev and Steven J. Miller and Liyang Zhang},
  title   = {Generalized More Sums Than Differences Sets},
  journal = {Journal of Number Theory},
  volume  = {132},
  number  = {5},
  pages   = {1054--1073},
  year    = {2012},
  doi     = {10.1016/j.jnt.2011.10.012},
  url     = {https://doi.org/10.1016/j.jnt.2011.10.012},
}

@article{tao_productset,
  author  = {Terence Tao},
  title   = {Product Set Estimates for Non-Commutative Groups},
  journal = {Combinatorica},
  volume  = {28},
  number  = {5},
  pages   = {547--594},
  year    = {2008},
  doi     = {10.1007/s00493-008-2271-7},
  url     = {https://doi.org/10.1007/s00493-008-2271-7},
}

@article{herrmann_small,
  author        = {Yorick Herrmann},
  title         = {A Classification of Small {MSTD} Sets in Arbitrary Fields},
  journal       = {arXiv preprint},
  year          = {2026},
  eprint        = {2608.22189},
  archivePrefix = {arXiv},
  primaryClass  = {math.NT},
  url           = {https://arxiv.org/abs/2608.22189},
}

@book{stein2005real,
  title={Real Analysis: Measure Theory, Integration, and Hilbert Spaces},
  author={Stein, Elias M. and Shakarchi, Rami},
  volume={3},
  series={Princeton Lectures in Analysis},
  year={2005},
  publisher={Princeton University Press},
  address={Princeton, NJ},
  isbn={978-0-691-11386-9}
}
